\documentclass{amsart}

\usepackage{latexsym}
\usepackage{amssymb, latexsym}
\usepackage{amsmath}
\usepackage{mathrsfs}
\usepackage{amsxtra}
\usepackage{amsthm}
\usepackage{verbatim}
\newtheorem{theorem}{Theorem}[section]
\newtheorem{corollary}[theorem]{Corollary}

\newtheorem{lemma}[theorem]{Lemma}

\newtheorem{claim}{Claim}[theorem]

\theoremstyle{definition}
\newtheorem{definition}[theorem]{Definition}
\newtheorem{question}[theorem]{Question}

\DeclareMathOperator{\Ult}{Ult}
\newcommand{\image}{\mathbin{\hbox{\tt\char'42}}}
\newcommand{\plus}{{+}}
\newcommand{\plusplus}{{{+}{+}}}
\newcommand{\Mbar}{{\overline{M}}}

\newcommand{\Union}{\bigcup}
\newcommand{\union}{\cup}
\newcommand{\of}{\subseteq}
\newcommand{\lt}[1]{{\smalllt}#1}
\newcommand{\lesseq}[1]{{\smallleq}#1}
\newcommand{\smallleq}{\mathrel{\mathchoice{\raise2pt\hbox{$\scriptstyle\leq$}}{\raise1pt\hbox{$\scriptstyle\leq$}}{\raise1pt\hbox{$\scriptscriptstyle\leq$}}{\scriptscriptstyle\leq}}}
\newcommand{\smalllt}{\mathrel{\mathchoice{\raise2pt\hbox{$\scriptstyle<$}}{\raise1pt\hbox{$\scriptstyle<$}}{\raise0pt\hbox{$\scriptscriptstyle<$}}{\scriptscriptstyle<}}}

\newcommand{\ltkappa}{{{\smalllt}\kappa}}
\newcommand{\leqkappa}{{{\smallleq}\kappa}}
\newcommand{\GCH}{{\rm GCH}}
\newcommand{\ORD}{\mathop{{\rm ORD}}}
\newcommand{\ZFC}{{\rm ZFC}}

\newcommand{\p}{\mathbb{P}}
\newcommand{\q}{\mathbb{Q}}
\newcommand{\la}{\langle}
\newcommand{\ra}{\rangle}
\newcommand{\her}[1]{H_{{#1}^+}}
\newcommand{\tail}{\text{tail}}
\newcommand{\forces}{\Vdash}

\newcommand{\pset}{\mathcal P}
\newcommand{\M}{\mathcal M}
\newcommand{\prop}{\mathscr P}
\newcommand{\restrict}{\upharpoonright}
\newcommand{\mless}{\mathrel{\vartriangleleft}}
\newcommand{\mgre}{\mathrel{\vartriangleright}}

\newcommand{\U}{\mathcal U}
\newcommand{\W}{\mathcal W}
\newcommand{\oR}{o_{\textup{Ram}}}
\newcommand{\oSR}{o_{\textup{stRam}}}
\newcommand{\oSuR}{o_{\textup{supRam}}}

\DeclareMathOperator{\dom}{dom}

\title{A Mitchell-like order for Ramsey and Ramsey-like cardinals}
\author{Erin Carmody}
\address[E. Carmody]{Nebraska Wesleyan University, Mathematics Department, 5000 Saint Paul Avenue Lincoln, NE 68504-2794}
\email{ecarmod2@NebrWesleyan.edu}

\author{Victoria Gitman}
\address[V. Gitman]{The City University of New York, CUNY Graduate Center, Mathematics Program, 365 Fifth Avenue, New York, NY 10016}
\email{vgitman@nylogic.org}
\urladdr{http://boolesrings.org/victoriagitman}

\author{Miha Habi\v c}
\address[M. Habi\v c] {The City University of New York, CUNY Graduate Center, Mathematics Program, 365 Fifth Avenue, New York, NY 10016}
\email{mhabic@gradcenter.cuny.edu}
\begin{document}
\begin{abstract}
Smallish large cardinals $\kappa$ are often
characterized by the existence of a collection of filters on $\kappa$, each of which is
an ultrafilter on the subsets of $\kappa$ of some transitive $\ZFC^-$-model of size $
\kappa$. We introduce a Mitchell-like order for
Ramsey and Ramsey-like cardinals, ordering such collections of small filters. We show
that the Mitchell-like order and the resulting notion of rank have all the desirable
properties of the Mitchell order on normal measures on a measurable cardinal. The
Mitchell-like order behaves robustly with respect to forcing constructions. We show that
extensions with cover and approximation properties cannot increase the rank of a Ramsey
or Ramsey-like cardinal. We use the results about extensions with cover and approximation
properties together with recently developed techniques about soft killing of
large-cardinal degrees by forcing to softly kill the ranks of Ramsey and Ramsey-like
cardinals.
\end{abstract}
\maketitle
\section{Introduction}
Mitchell introduced the \emph{Mitchell order} on normal measures on a measurable cardinal
$\kappa$ in~\cite{mitchell:mitchellRank}, where he defined that $U\mless W$ for two
normal measures $U$ and $W$ on $\kappa$ whenever $U\in\Ult(V,W)$, the ultrapower of the
universe $V$ by $W$. Since $\mless$ is easily seen to be well-founded, we can define the
ordinal rank $o(U)$ of a normal measure and define $o(\kappa)$, the Mitchell rank of
$\kappa$, to be the supremum of $o(U)$ over all normal measures $U$ on $\kappa$. The
Mitchell rank of $\kappa$ tells us to what extent measurability is reflected below
$\kappa$.
Mitchell used the Mitchell order to study coherent sequences of normal measures,
which allowed him to generalize Kunen's $L[U]$ construction to canonical inner models
with many measures (cf.~\cite{mitchell:mitchellRank}). The Mitchell rank of a measurable
cardinal has also proved instrumental in calibrating consistency strength of set
theoretic assertions. Gitik showed, for instance, that the consistency strength of a
measurable cardinal at which the $\GCH$ fails is a measurable cardinal $\kappa$ with
$o(\kappa)=\kappa^\plusplus$ (cf.~\cite{gitik:measurablenotCH}).
The notion of Mitchell order
generalizes to extenders, where it has played a role in constructions of core models.

In this article, we introduce a Mitchell-like order for Ramsey and Ramsey-like cardinals.
Although we tend to associate smaller large
cardinals $\kappa$ with combinatorial definitions, many of them have characterizations in
terms of existence of elementary embeddings. The domains of these embeddings are
\emph{weak $\kappa$-models}, transitive models of $\ZFC^-$ of size $\kappa$ and height
above $\kappa$, or
some stronger version of these such as \emph{$\kappa$-models}, which are additionally
closed under $\ltkappa$-sequences. Usually, the embeddings are ultrapower or extender
embeddings by mini-measures or mini-extenders that apply only to the $\kappa$-sized
domain of the embedding. If $M$ is a transitive model of $\ZFC^-$ and $\kappa$ is a
cardinal in $M$, then we call $U\subseteq \pset(\kappa)^M$ an
$M$-ultrafilter\footnote{The notation $\pset(\kappa)^M$ is meant to denote
\(\pset(\kappa)\cap M\), whether or not this is actually an
element of $M$.} if it is an ultrafilter on $\pset(\kappa)^M$ that is normal for
sequences in $M$. In most interesting cases, an $M$-ultrafilter is external to $M$, but
we can still form the ultrapower by using functions on $\kappa$ that are elements of  $M$.
A prototypical characterization of a smaller large cardinal $\kappa$ states that every $A
\subseteq\kappa$ is an element of a weak $\kappa$-model $M$ (with additional requirements) for
which there is an $M$-ultrafilter on $\kappa$ (with additional requirements). The
additional requirements on $M$ and the $M$-ultrafilter are dictated by the large cardinal
property. The simplest such characterization belongs to weakly compact cardinals, where
there is the minimal requirement on the $M$-ultrafilter, namely that the ultrapower of
$M$ is well-founded.

Given a large-cardinal property $\prop$ with an embedding characterization as discussed
above (such as weak compactness, Ramseyness, etc.), let us say that an $M$-ultrafilter is
a \emph{$\prop$-measure} if it, together with $M$, witnesses $\prop$ and that a
$\prop$-measure is \emph{$A$-good} for some $A\subseteq\kappa$ if $A\in M$.\footnote{For
technical reasons we also require that $V_\kappa\in M$. Note that if $M$ is a
$\kappa$-model, then $V_\kappa\in M$ follows.}
To avoid having to specify which model $M$ we associate to a given $\prop$-measure $U$,
we will always associate it with a unique minimal model $M_U$, namely the
$\her{\kappa}$ of any such model \(M\). Let us say that a collection $\U$ of $\prop$-
measures is a \emph{witness} for $\prop$ if for every $A\subseteq\kappa$, it contains
some $A$-good $\prop$-measure. So while a normal measure on $\kappa$ witnesses the
measurability of $\kappa$, a witness collection of $\prop$-measures is precisely what
witnesses $\prop$ for one of these smaller large cardinals. This suggests that a
reasonable Mitchell-like order should not be comparing the tiny $\prop$-measures, but
rather witness collections of $\prop$-measures in a way that ensures that the
corresponding rank $o_{\prop}(\kappa)$ of $\kappa$ measures the extent to which $\prop$
is reflected below $\kappa$. We will call this order the M-order in honor of Mitchell.

\begin{definition}[M-order]
Suppose that $\kappa$ has a large-cardinal property $\prop$ with an
embedding characterization. Given two witness collections $\U$ and $\W$ of
$\prop$-measures, we define that $\U\mless \W$ if
\begin{enumerate}
\item For every $W\in \W$ and $A\subseteq\kappa$ in the ultrapower $N_W$ of $M_W$ by $W$,
there is an $A$-good $U\in\U\cap N_W$ such that \(N_W\models\text{``\(U\) is an \(A\)-good
\(\prop\)-measure on \(\kappa\)''}\).

\item $\mathcal U\subseteq \Union_{W\in\W}N_W$.
\end{enumerate}
\end{definition}

The key part of the definition is clause (1). It states that the elements of \(\U\)
witness that \(\kappa\) retains the property \(\prop\)
in the ultrapowers by the elements of \(\W\). It is tempting to say that \(\U\) itself
should witness \(\prop\) in those ultrapowers, but note that \(\U\) is too large
to be an element of a weak \(\kappa\)-model. Clause (2) ensures that the collections
of under consideration do not contain superfluous \(\prop\)-measures.

Mitchell proved that Ramsey cardinals have an embedding characterization and Gitman used
generalizations of it to define the Ramsey-like cardinals: \emph{$\alpha$-iterable},
\emph{strongly Ramsey}, and \emph{super Ramsey cardinals} (cf.~\cite{mitchell:ramsey} and
\cite{gitman:ramsey}).
Thus, a \emph{Ramsey measure} $U$ is a weakly amenable $\omega_1$-intersecting $M_U$-
ultrafilter, a \emph{strongly Ramsey measure} $U$ is a weakly amenable $M_U$-ultrafilter,
where $M_U$ is a $\kappa$-model, and a \emph{super Ramsey measure} is a weakly amenable
$M_U$-ultrafilter where $M_U$ is a $\kappa$-model elementary in $\her{\kappa}$.

We will show that the M-order and the corresponding notion of M-rank share all the
desirable features of the Mitchell order on normal measures on a measurable cardinal.
For example, the order is transitive and well-founded. Note that since an ultrapower of
a weak $\kappa$-model has size at most $\kappa$, the M-rank of a large cardinal
\(\kappa\) can be at most $\kappa^\plus$, in contrast with the upper bound of
\((2^\kappa)^+\) in the case of the usual Mitchell rank for a measurable cardinal.

\begin{theorem}\label{thm:RankInUltrapower}
Suppose $\U$ is a witness collection of $\prop$-measures, where $\prop$ is Ramsey or Ramsey-like, such that $o_{\prop}(\U)\geq\alpha$. Then:
\begin{enumerate}
\item For every $U\in \U$, the ultrapower $N_U$ of $M_U$ by $U$ satisfies $o_{\prop}(\kappa)\geq\alpha$.
\item There is a good collection $\W$ with $o_{\prop}(\W)=\alpha$ such that
$N_W\models o_{\prop}(\kappa)=\alpha$ for all $W\in\W$.
\end{enumerate}
\end{theorem}

\noindent We should not expect an analogue of Theorem~\ref{thm:RankInUltrapower}~(1) with
equality because we are now dealing with collections of measures instead of a single
measure and so Theorem~\ref{thm:RankInUltrapower}~(2) is the best possible result.

\begin{theorem}
Any strongly Ramsey cardinal $\kappa$ has the maximum Ramsey M-rank
$\oR(\kappa)=\kappa^\plus$, any super Ramsey cardinal $\kappa$ has the maximum strongly
Ramsey M-rank $\oSR(\kappa)=\kappa^\plus$, and any measurable cardinal $\kappa$ has the
maximum super Ramsey M-rank $\oSuR(\kappa)=\kappa^\plus$.
\end{theorem}

We will show that the new Mitchell order behaves robustly with respect to forcing
constructions. We show that extensions with cover and approximation properties cannot
create new Ramsey or Ramsey-like cardinals or increase their M-rank. Hamkins showed,
in~\cite{hamkins:coverandapproximations}, that most large cardinals cannot be created in
extensions with cover and approximation properties and we provide several modifications
of his techniques to the embeddings characterizing Ramsey and Ramsey-like cardinals. This
result is of independent interest since it was not previously known whether Ramsey
cardinals can be created in extensions with cover and approximation properties.

\begin{theorem}
If $V\subseteq V'$ has the $\delta$-cover and $\delta$-approximation properties for some
regular cardinal $\delta<\kappa$ of $V'$, then
$o_{\prop}^V(\kappa)\geq o_{\prop}^{V'}(\kappa)$,
where $\prop$ is strongly or super Ramsey, and if we additionally assume that
$V^\omega\subseteq V$ in $V'$, then $\oR^V(\kappa)\geq \oR^{V'}(\kappa)$.
\end{theorem}

Using the results about extensions with approximation and cover properties together with
new techniques recently developed in Carmody's dissertation \cite{carmody:thesis} about
softly killing degrees of large cardinals with forcing, we show how to softly kill the
M-rank of a Ramsey or Ramsey-like cardinal by forcing.

\begin{theorem}
If $\kappa$ has $o_{\prop}(\kappa)=\alpha$, where $\prop$ is Ramsey or Ramsey-like and $
\beta<\alpha$, then there is a cofinality-preserving forcing extension in
which $o_{\prop}(\kappa)=\beta$.
\end{theorem}

Although the general framework of the M-order we have sketched here applies to many
smallish large cardinals, we focus in this paper on its application to Ramsey, strongly
Ramsey and super Ramsey cardinals. Other instances of it definitely warrant further
research.

\section{Preliminaries}\label{sec:prelim}
\subsection{Properties of Ramsey and Ramsey-like cardinals}
As discussed in the introduction, many large cardinals $\kappa$ below a measurable
cardinal have the prototypical characterization, where for every $A\subseteq\kappa$,
there is a weak $\kappa$-model $M$, with some additional properties, containing $A$ for
which there is an $M$-ultrafilter on $\kappa$, with some additional properties, where the
additional properties are what distinguishes the different large cardinal properties.
Formally, if $M$ is a transitive model of $\ZFC^-$, then an $M$-ultrafilter is a
collection $U\subseteq \pset(\kappa)^M$ such that the structure $\la M,\in, U\ra$, with a
predicate for $U$, satisfies that $U$ is a normal ultrafilter on $\kappa$.\footnote{We
will always assume that an ultrafilter on a cardinal $\kappa$ contains the tail sets and
therefore is non-principal. It also follows from this assumption that an $M$-ultrafilter
is $\kappa$-complete for sequences in $M$.} We can form the ultrapower of a model $M$ by
an $M$-ultrafilter using functions on $\kappa$ that are elements of $M$, but the
ultrapower may not be well-founded. Let us call \(U\), an $M$-ultrafilter,
\emph{$\delta$-intersecting} for a cardinal $\delta$ if every collection of fewer than
\(\delta\) many sets from it has non-empty intersection.\footnote{In the literature
such $M$-ultrafilters are often called $\delta$-complete which we find confusing because $\delta$-complete ultrafilters are supposed to have the property that the intersection of fewer than $\delta$-many sets in the ultrafilter is itself in the ultrafilter. But in the situation of $M$-ultrafilters, the intersection may not even be an element of $M$.}
Standard arguments show that the ultrapower of a model $M$ by an $\omega_1$-intersecting
$M$-ultrafilter is well-founded, but, in contrast with the case of actual ultrafiters,
this condition is not necessary.

Many set theoretic constructions use iterated ultrapowers by a measure on a measurable
cardinal. If $U$ is an ultrafilter on some set, then the ultrapower (of $V$) construction
with it can be iterated along the ordinals by taking the ultrapower by the image of the
previous stage's ultrafilter at successor stages and direct limits at limit stages.
Gaifman showed that if an ultrafilter is countably complete, which is equivalent to
having a well-founded ultrapower, then all its iterated ultrapowers are well-founded
(cf.~\cite{gaifman:ultrapowers}). For an $M$-ultrafilter, to be able to define the
successor stage ultrafilters in the iterated ultrapower construction, it must be at least
partially internal to $M$, a property that is captured by the notion of weak amenability.
An $M$-ultrafilter is \emph{weakly amenable} if for every $X\in M$ of size at most $
\kappa$ in $M$, the intersection $X\cap U$ is in $M$.\footnote{The property is a
weakening of the usual definition of amenability because we restrict to $X$ of size at
most $\kappa$ in $M$.} Although weak amenability allows us to define all the iterated
ultrapowers, it does not have any bearing on their well-foundedness. Kunen showed that
being $\omega_1$-intersecting is sufficient for well-foundedness
(cf.~\cite{kunen:ultrapowers}), but it is not necessary.
Unlike measures on $\kappa$, where
either all the iterated ultrapowers are well-founded or none are, we will see below that
it is consistent to have $M$-ultrafilters with exactly $\alpha$-many well-founded
iterated ultrapowers for any countable ordinal $\alpha$.\footnote{It follows from
Gaifman's arguments in \cite{gaifman:ultrapowers} for ultrapowers by a measure that an
$M$-ultrafilter with $\omega_1$-many well-founded iterated ultrapowers, already has all
well-founded iterated ultrapowers.}

Recall that $\kappa$ is weakly compact if and only if $\kappa^{\ltkappa}=\kappa$ and
every $A\subseteq\kappa$ is contained in a weak $\kappa$-model for which there is an
$M$-ultrafilter on $\kappa$ with a well-founded ultrapower. This characterization can be
strengthened in a number of significant ways. For instance, we can assume that $M$ is a
$\kappa$-model that is elementary in $\her{\kappa}$ and hence reflects $V$ to a certain
extent. In fact, we can assume that \emph{every} weak $\kappa$-model $M$ has an
$M$-ultrafilter with a well-founded ultrapower. Surprisingly, we cannot strengthen the
characterization of Ramsey cardinals in the same fashion. Recall now that $\kappa$ is
Ramsey if and only if every $A\subseteq\kappa$ is contained in a weak $\kappa$-model $M$
for which there is a weakly amenable $\omega_1$-intersecting $M$-ultrafilter on $\kappa$.
If we strengthen the characterization to say that $M$ is a $\kappa$-model, then we get what
Gitman calls a \emph{strongly Ramsey} cardinal, which she showed is a stationary limit of
Ramsey cardinals.\footnote{Note that a weakly amenable $M$-ultrafilter for a $\kappa$-model
$M$ is automatically $\omega_1$-intersecting.}
Indeed, even if we just assume that $M$ is closed under countable sequences, then we already
get a large cardinal, call it $\omega$-\emph{closed Ramsey}, that is a stationary limit of
Ramsey cardinals and, which, as we will see in Section~\ref{sec:MitchellOrder}, has the
maximum Ramsey M-rank. Strongly Ramsey cardinals can also be viewed as quite strong because
they are limits of the completely Ramsey cardinals defined by Feng
in~\cite{feng:completelyRamsey}. If we further strengthen the characterization to say that
$M$ is elementary in $\her{\kappa}$, we get a \emph{super Ramsey} cardinal that is in turn a
stationary limit of strongly Ramsey cardinals (but weaker than a measurable cardinal).
Assuming that \emph{every} weak $\kappa$-model $M$ has a weakly amenable
$\omega_1$-intersecting $M$-ultrafilter turns out to be inconsistent. For details,
see~\cite{gitman:ramsey}.

The requirement that the $M$-ultrafilters are weakly amenable already takes us well beyond
weak compactness. If every $A\subseteq\kappa$ is contained in a weak $\kappa$-model $M$ for
which there is a weakly amenable $M$-ultrafilter with a well-founded ultrapower, then
$\kappa$ is a stationary limit of completely ineffable cardinals, which sit atop a hierarchy
of ineffability. The following is a very useful characterization of weak amenability. If $U$
is an $M$-ultrafilter on $\kappa$ and $j:M\to N$ is the ultrapower by $U$, then $U$ is
weakly amenable if and only if $\pset(\kappa)^M=\pset(\kappa)^N$. Moreover, if $j:M\to N$ is
any embedding with critical point $\kappa$ and $\pset(\kappa)^M=\pset(\kappa)^N$, then the
$M$-ultrafilter $U$ obtained from $j$ in the usual way, is weakly amenable. We can stratify
weakly amenable $M$-ultrafilters by degrees of iterability. Let us say that an
$M$-ultrafilter is $\alpha$-\emph{iterable} if it has $\alpha$-many well-founded iterated
ultrapowers and that it is \emph{iterable} if it is $\omega_1$-iterable. Gitman defined that
a cardinal $\kappa$ is $\alpha$-iterable (for $1\leq\alpha\leq\omega_1$) if every
$A\subseteq\kappa$ is contained in a weak $\kappa$-model $M$ for which there is a weakly
amenable $\alpha$-iterable $M$-ultrafilter. Gitman and Welch showed that the
$\alpha$-iterable cardinals form a hierarchy of strength (cf.~\cite{gitman:welch}) and
Sharpe and Welch showed that an $\omega_1$-Erd\H os cardinal is a limit of
$\omega_1$-iterable cardinals (cf.~\cite{welch:ramsey}).

Suppose that $M$ is a weak $\kappa$-model and $U$ is an $M$-ultrafilter. Consider the
submodel $\Mbar=\her{\kappa}^M$ consisting of all sets that have hereditary size
at most $\kappa$ in \(M\). Clearly $\Mbar$ is itself a weak $\kappa$-model and if $M$ was a
$\kappa$-model then $\Mbar$ is as well. Also, $U$ is an $\Mbar$-ultrafilter and it
retains all other relevant properties with respect to $\Mbar$ that it had with respect to
$M$, such as being weakly amenable, $\alpha$-iterable, or $\omega_1$-intersecting. The model
$\Mbar$ is the unique minimal model $M_U$ for $U$ we discussed in the introduction. Note
that the sets in $M_U$ are precisely the Mostowski collapses of well-founded binary
relations on $\kappa$ coded by sets in $U$ together with their complements, so that $M_U$
can be recovered from $U$ in any model of a sufficient fragment of set theory. In future
arguments, we will only consider $M_U$-ultrafilters $U$, where $V_\kappa\in M_U$ so that the
ultrapower $N_U$ thinks that $V_{j(\kappa)}$ exists. For such a weakly amenable \(U\),
$M_U$ is precisely $\her{\kappa}^{N_U}$ and so must be an element of $N_U$.

It turns out that if $U$ is an iterable $M_U$-ultrafilter, then $U$ also codes a weak
$\kappa$-model $M^*_U$ of full $\ZFC$ so that it is also an iterable $M^*_U$-ultrafilter.
Specifically, we can take $M_U^*=V_{j(\kappa)}^{N_U}$, where $j:M_U\to N_U$ is the
ultrapower map of $M_U$ by $U$. The ultrapower map of $M_U^*$ by $U$ is the restriction
$j^*:M_U^*\to N_U^*$ of the ultrapower of $N_U$ by $U$. If $M_U$ was a $\kappa$-model, then
so is $M^*_U$. The embedding $j^*$ has several useful properties, such as
$M_U=\her{\kappa}^{M_U^*}$, that
$M^*_U=V_{j^*(\kappa)}^{N_U^*}$ is in $N^*_U$, and $M_U^*\prec N_U^*$. The same construction
cannot be carried out with a partially iterable $M_U$-ultrafilter because the iterability of
$U$ decreases when you pass to the model $M_U^*$. Indeed, assuming that there are
$\alpha$-iterable ultrafilters for models of $\ZFC$ produces a stronger notion than an
$\alpha$-iterable cardinal.

For forcing constructions with Ramsey cardinals, which we discuss below, we will need to
make some additional assumptions on the weak $\kappa$-model $M$. We define that a weak
$\kappa$-model $M$ is $\omega$-\emph{special} if it is the union of an elementary chain of
(not necessarily transitive) substructures
\[
\kappa\in M_0\prec M_1\prec \cdots\prec M_n\prec\cdots
\]
for $n<\omega$ such that each $M_n\in M$ and $|M_n|^M=\kappa$.\footnote{In fact we will not
require the elementary chain requirement in the definition.}
The ultrapower $N$ of an
$\omega$-special weak $\kappa$-model $M$ by a weakly amenable $M$-ultrafilter on $\kappa$ is
$\omega$-special as witnessed by the sequence $\la X_n\mid n<\omega\ra$, where
\begin{displaymath}
X_n=\{j(f)(\kappa)\mid f:\kappa\to M_n, f\in M_n\},
\end{displaymath}
and if $M=M_U$ then $M_U^*$ is $\omega$-special as well. (See Lemmas~2.7 and 2.9 of
\cite{gitmancody:eastonramsey}.)

\begin{lemma}\label{lem:omegaspecial}
Suppose $\kappa$ is Ramsey. Then every $A\subseteq\kappa$ is contained in an
$\omega$-special weak $\kappa$-model $M$ for which there is a weakly amenable
$\omega_1$-intersecting $M$-ultrafilter on $\kappa$.
\end{lemma}

\begin{proof}
Fix $A\subseteq\kappa$ and choose some weak $\kappa$-model $\overline M$ containing $A$ for
which there is a weakly amenable $\omega_1$-intersecting $\overline M$-ultrafilter $\overline U$ on
$\kappa$. Let $\overline N$ be the ultrapower of $\overline M$ by $\overline U$. We can
assume that $\kappa$ is the largest cardinal of $\overline M$ and therefore
$\overline M=\her{\kappa}^{\overline N}$ is an element of $\overline N$. Working in
$\overline N$, let $M_0$ be any transitive elementary submodel of $\her{\kappa}$ of size
$\kappa$ with $A\in M_0$. Since $\overline N$ and $\overline M$ have the same subsets of
$\kappa$, $M_0$ and hence $U_0=M_0\cap \overline U$ are in $\overline M$. So working in
$\overline N$, we can choose a transitive $M_1\prec \her{\kappa}$ of size $\kappa$ with
$M_0, U_0\in M_1$. Continuing in this fashion, we obtain a sequence
$\la (M_n,U_n)\mid n<\omega\ra$. Let $M=\Union_{n<\omega}M_n$ and $U=\Union_{n<\omega}U_n$.
The model $M$ is $\omega$-special as witnessed by the sequence $\la M_n\mid n<\omega\ra$
(the $M_n$ are even transitive) and $U=M\cap \overline U$ is $\omega_1$-intersecting and
weakly amenable by construction.
\end{proof}

Since \(M_U^*\) is \(\omega\)-special whenever $M_U$ is, it follows that if $\kappa$ is
Ramsey, then every $A\subseteq\kappa$ is contained in an $\omega$-special weak
$\kappa$-model $M\models\ZFC$ for which there is a weakly amenable $\omega_1$-intersecting
$M$-ultrafilter on $\kappa$.

\subsection{Forcing constructions}
Suppose $\p\subseteq V_\kappa$ is a poset and we would like to verify that $\kappa$ is
Ramsey in a forcing extension $V[G]$ by $\p$. Since every $A\subseteq\kappa$ in $V[G]$ has a
$\p$-name $\dot A$ of hereditary size $\kappa$ in $V$, $\dot A$ together with $\p$ can be
put into a weak $\kappa$-model $M$, which ensures that $A$ is in the weak $\kappa$-model
$M[G]$. Thus, it suffices to show that every ultrapower $j:M\to N$ of a weak $\kappa$-model
$M$ by a weakly amenable $\omega_1$-intersecting $M$-ultrafilter can be lifted to
$j:M[G]\to N[j(G)]$ so that the lift is the ultrapower by a weakly amenable
$\omega_1$-intersecting $M[G]$-ultrafilter in $V[G]$. The \emph{lifting criterion} states
that $j$ lifts to $j:M[G]\to N[H]$ with $H=j(G)$ if and only if $j\image G\subseteq H$.
In this setting, when constructing a generic filter, we usually work with a \(\kappa\)-model
\(M\) and a poset \(\p\) that is $\lt\kappa$-closed in $M$. This suffices for the existence
of an \(M\)-generic filter for \(\p\). Instead of this approach, which does not apply to
weak \(\kappa\)-models, we will use the following
\emph{diagonalization criterion}, introduced in \cite{gitman:ramseyindes}.

\begin{lemma}[Diagonalization criterion, \cite{gitman:ramseyindes}]\label{le:diag2}
If $M$ is an $\omega$-special weak $\kappa$-model and $\p$ is a $\leqkappa$-distributive
poset in $M$, then there is an $M$-generic filter for $\p$.
\end{lemma}

\noindent The lift of an ultrapower embedding is always an ultrapower embedding by the
ultrafilter $W$ obtained from the lift and we will usually use a direct argument to verify
that $M[G]$ and $N[j(G)]$ have the same subsets of $\kappa$ (which demonstrates weak
amenability). To show that $W$ is $\omega_1$-intersecting, we have the following lemma.

\begin{lemma}[Gitman, Johnstone \cite{gitman:ramseyindes}]\label{le:ramseycountablyclosed}
Suppose that $M$ is a weak $\kappa$-model and $j:M\to N$ is the ultrapower by an
$\omega_1$-intersecting $M$-ultrafilter on $\kappa$. Suppose further that $\p\in M$ is a
countably closed forcing notion and $G\of\p$ is $M$-generic. If the ultrapower map $j$ lifts
to $j:M[G]\to N[j(G)]$, then the lift $j$ is the ultrapower map by an
$\omega_1$-intersecting $M[G]$-ultrafilter in $V[G]$.
\end{lemma}

\section{The M-order for Ramsey and Ramsey-like cardinals}\label{sec:MitchellOrder}
Let's start by recalling and making more precise the definitions we made in the
introduction.
Let us say that $U\subseteq \pset(\kappa)$ of size $\kappa$ is a \emph{small universe
measure} if the collection of Mostowski collapses of well-founded binary relations on
$\kappa$ coded by sets in $U$ and their complements is a weak $\kappa$-model $M_U$ such that
$V_\kappa\in M_U$ and $U$ is an $M_U$-ultrafilter. We
will write $N_U$ for the Mostowski collapse of the ultrapower of $M_U$ by $U$, provided that
the ultrapower is well-founded. We say that a small universe measure $U$ is a \emph{Ramsey
measure} if it is weakly amenable (to $M_U$) and $\omega_1$-intersecting, we say that $U$ is
a \emph{strong Ramsey measure} if it is weakly amenable and $M_U$ is a $\kappa$-model, and
we say that $U$ is a \emph{super Ramsey measure} if it is a strong Ramsey measure and
$M_U\prec\her{\kappa}$. We will carry out all the arguments below for Ramsey cardinals
since they are the most complicated, pointing out at the end that analogous or simpler
arguments work for strongly Ramsey and super Ramsey cardinals. The interested reader can
note along the way where the arguments adapt to other smallish large cardinals, which we do
not discuss here.

A first approach to defining a Mitchell-like order for Ramsey cardinals might be to consider
ordering Ramsey measures on a fixed Ramsey $\kappa$ analogously to the Mitchell order on
normal measures on a measurable cardinal.

\begin{definition}
Given two Ramsey measures $U$ and $W$ on a cardinal $\kappa$, define that $U\mless W$ if $U\in N_W$.
\end{definition}

\begin{lemma}\label{lem:measureOrder}
The relation \(\mless\) on Ramsey measures on a cardinal $\kappa$ is transitive and
well-founded.
\end{lemma}

\begin{proof}
Transitivity is straightforward. To see that the relation is also well-founded, notice
that \(U\mless W\) implies \(j_U(\kappa)<j_W(\kappa)\), where \(j_U\) and \(j_W\) are
the ultrapower maps with respect to \(U\) and \(W\). This is so because if \(U\in N_W\)
then also \(j_U,N_U\in N_W\) and \(N_W\models |U|=|N_U|\). Now,
\(|U|^{N_W}\leq (2^\kappa)^{N_W}<j_W(\kappa)\) since \(j_W(\kappa)\) is inaccessible in
\(N_W\). Since \(j_U(\kappa)\) is an element of \(N_U\), a transitive substructure of
\(N_W\), we must have \(j_U(\kappa)<j_W(\kappa)\). The well-foundedness of \(\mless\)
now follows since an infinite decreasing chain of Ramsey measures would yield
an infinite decreasing chain of ordinals.
\end{proof}

We should also notice that no Ramsey measure \(U\) can have more than \(\kappa\)
predecessors in the order \(\mless\), since the ultrapower \(N_U\) has cardinality
\(\kappa\).

The order $\mless$ on Ramsey measures is an interesting object in its own right, but it is
not useful for defining degrees of Ramsey cardinals with the intention to capture the extent
to which Ramseyness is reflected below $\kappa$ because if $\kappa$ is Ramsey then there
are already Ramsey measures of all possible ranks $\alpha<\kappa^+$.

\begin{lemma}
If $\kappa$ is Ramsey and $\alpha<\kappa^+$, then there is a Ramsey measure on $\kappa$ of
rank at least $\alpha$ in the $\mless$-order.
\end{lemma}

\begin{proof}
Suppose inductively that for all $\beta<\alpha$, there is a Ramsey measure $U_\beta$ on
$\kappa$ whose rank in the $\mless$-order is at least $\beta$. Let $U$ be some Ramsey
measure on $\kappa$ such that $\{U_\beta\mid\beta<\alpha\}\subseteq M_U$,
which is possible since this set has hereditary size $\kappa$. Clearly the rank of $U$
in $\mless$ is at least $\alpha$.
\end{proof}

\noindent The same analysis holds for strongly Ramsey and super Ramsey measures. It is to be
expected that ordering the small universe measures is not the right analogue of the Mitchell
order for smaller large cardinals because such a cardinal is characterized by the existence
of many and not just one such measure. This brings us back to the definition of the M-order
on collections of small universe measures, which we restate here again in full generality
before going to back concrete arguments for Ramsey cardinals. Recall that a $\prop$-measure
(for some large cardinal property $\prop$ such as Ramsey, strongly Ramsey, etc.) is
$A$-\emph{good} for $A\subseteq\kappa$ if
$A\in M_U$. We call a collection of $\prop$-measures a \emph{witness collection} if it
contains at least one $A$-good $\prop$-measure for every $A\subseteq\kappa$. Then
$\kappa$ has property $\prop$ if and only if there is a witness collection of
$\prop$-measures.

\begin{definition}(M-order)
Suppose that $\kappa$ is a large cardinal with property $\prop$ having a suitable embedding
characterization. Given two witness collections $\U$ and $\W$ of $\prop$-measures, we define
that $\U\mless \W$ whenever
\begin{enumerate}
\item For every $W\in \W$ and $A\subseteq\kappa$ in the ultrapower $N_W$ of $M_W$ by $W$,
there is an $A$-good $U\in\U\cap N_W$ such that \(N_W\models\text{``\(U\) is an \(A\)-good
\(\prop\)-measure on \(\kappa\)''}\).

\item $\U\subseteq \Union_{W\in\W}N_W$.
\end{enumerate}
\end{definition}

\begin{lemma}
The $\mless$-order on witness collections of Ramsey measures on a cardinal $\kappa$ is
transitive and well-founded.
\end{lemma}

\begin{proof}
First, we show transitivity. Suppose that $\U,\,\W,\,\mathcal Z$ are witness collections of
Ramsey measures on $\kappa$ such that $\U\mless \W$ and $\W\mless\mathcal Z$. We have
$\U\subseteq\Union_{W\in \W} N_W\subseteq\Union_{Z\in\mathcal Z}N_Z$, where the first
inclusion follows by definition and second inclusion follows because
$\W\subseteq \Union_{Z\in\mathcal Z}N_Z$ and if $W\in \W$ is in some $N_Z$, then $N_Z$ has
$N_W$ as well. This verifies requirement (2) in showing that $\U\mless \mathcal Z$, and now
we verify (1). If $Z\in \mathcal Z$ and $A\in N_Z$, then since $\W\mless \mathcal Z$, there
is $W\in N_Z\cap \W$ with $A\in N_W$, but then, since $\U\mless \W$, there is
$U\in N_W\cap \U\subseteq N_Z\cap \U$ with $A\in N_U$. Thus, $\U\mless \mathcal Z$.

Next, suppose towards a contradiction that $\mless$ is ill-founded for witness collections of
Ramsey measures on $\kappa$ and fix a $\mless$-descending sequence
\begin{displaymath}
\U_0\mgre\U_1\mgre\cdots\mgre\U_n\mgre\cdots
\end{displaymath}
of witness collections. Let $U_0$ be any element of $\U_0$. Since $\U_1\mless \U_0$, then
$N_{U_0}$ has some element $U_1$ of $\U_1$, and so $U_1\mless U_0$ in the ordering on Ramsey
measures. Continuing in the same manner, we obtain a descending sequence
\begin{displaymath}
U_0\mgre U_1\mgre\cdots\mgre U_n\mgre\cdots
\end{displaymath}
in the $\mless$-order on Ramsey measures, which is impossible by
Lemma~\ref{lem:measureOrder}.
\end{proof}

The lemma implies that we can assign to each witness collection \(\U\) of Ramsey measures
on \(\kappa\) its rank \(\oR(\U)\) in the order \(\mless\). We can then let
\[
\oR(\kappa)=\{\oR(\U)\mid\text{\(\U\) is  a witness collection of Ramsey measures on
\(\kappa\)}\}.
\]
We define ranks for strongly and super Ramsey cardinals in similar fashion.

The defining property of the Mitchell order is that a normal measure $U$ on $\kappa$ has
rank $\alpha$ if and only if $\Ult(V,U)$ satisfies $o(\kappa)=\alpha$. The analogous result
for the M-order on witness collections of Ramsey measures on $\kappa$ will be that
$\oR(\U)\geq\alpha$ if and only if for every $U\in\U$, $N_U\models \oR(\kappa)\geq\alpha$.
It is not feasible to obtain equality because a witness collection $\U$ of rank $\alpha$ can
easily have elements $U$ with $N_U\models \oR(\kappa)>\alpha$. Still we will be able to
show that ``well-behaved" collections always exist: if $\alpha<\oR(\kappa)$, then there is
some witness collection $\W$ with $\oR(\W)=\alpha$ such that
$N_W\models \oR(\kappa)=\alpha$ for every $W\in \W$.

A subtle issue that arises when trying to prove that the rank of a witness collection
corresponds to the rank of $\kappa$ in the ultrapowers of its measures in the case of Ramsey
cardinals (but does not arise for strongly Ramsey or super Ramsey cardinals) is that the
ultrapower of a weak $\kappa$-model can already be wrong about whether something is a Ramsey
measure or not since in most cases it is missing countable sequences. To prove the result we
will temporarily use a stronger notion of Ramsey measure, which will be absolute for
transitive $\ZFC^-$-models, and show that the two notions give the same M-rank.

Let us say that an $M_U$-ultrafilter $U$ on $\kappa$ is a \emph{certified Ramsey measure} if
it is weakly amenable and there is some unbounded $I\subseteq\kappa$ such that $X\in U$ if
and only if \(X\) contains a tail of $I$, in which case we say that $I$ \emph{certifies}
$U$. Clearly, every certified Ramsey measure is a Ramsey measure because it is even
$\kappa$-intersecting (every sequence of $\lt\kappa$-many sets in $U$ has a non-empty
intersection) and certified
Ramsey measures have the advantage of being absolute between transitive models of set
theory. In fact, a standard proof that Ramsey cardinals have Mitchell's characterization
(see \cite{dodd:coremodel} or a more detailed exposition in \cite{gitman:dissertation})
actually produces for every $A\subseteq\kappa$, an $A$-good certified Ramsey measure.
Briefly, the proof uses the notion of a \emph{good set of indiscernibles} for a structure
$L_\kappa[A]$ with $A\subseteq\kappa$ and shows that if $\kappa$ is Ramsey, then for every
$A\subseteq\kappa$, there is a good set $I$ of indiscernibles for $L_\kappa[A]$ of size
$\kappa$. The indiscernibles in $I$ are then used to construct a weak $\kappa$-model $M$
(with the largest cardinal $\kappa$) and a weakly amenable $\omega_1$-intersecting
$M$-ultrafilter that is certified by $I$.
Given a witness collection $\U$ of certified
Ramsey measures, let $\oR^*(\U)$ denote the rank of $\U$ in the $\mless$-order restricted to
witness collections of certified Ramsey measures and let $\oR^*(\kappa)$ be the supremum of
the ranks of all such $\U$.

For inductive arguments about the M-rank of Ramsey or Ramsey-like cardinals, we will often
need to know that the results hold not just in $V$, but more generally in transitive set
models of $\ZFC^-$ that know enough about the cardinal. If $\M$ is a transitive model of
$\ZFC^-$ and $\kappa$ is a cardinal in $\M$, we will say that \(\M\) is
\emph{practical} for \(\kappa\) if $V_{\kappa+3}^{\M}$ exists (this ensures that the model
can put together all witness collections in order to rank them).
We formulate the following several results about Ramsey and Ramsey-like cardinals for
practical models and note that the results also hold for \(V\), since we may always work
with a sufficiently large \(H_\lambda\) for which all of the notions are absolute.

\begin{lemma}\label{lem:rankCertifiedMeasures}
If \(\alpha\) is an ordinal and $\M$ is practical for $\kappa$, then, in $\M$,
a witness collection $\U$ of certified Ramsey measures on $\kappa$ has $\oR^*(\U)\geq\alpha$
if and only if $N_U\models\oR^*(\kappa)\geq\alpha$ for all $U\in\U$.
\end{lemma}

\begin{proof}
We will argue by induction on $\alpha$. The case $\alpha=0$ is trivial. So suppose that the
statement is true for all $0\leq \beta<\alpha$.  Fix an \(\M\) practical for \(\kappa\)
and work in $\M$.

In one direction, fix a witness collection $\U$ of certified measures on $\kappa$ with
$\oR^*(\U)\geq\alpha$. Let $U\in\U$. We must show that $N_U$ has witness collections of
certified Ramsey measures of all ranks $\beta<\alpha$. Fix $\beta<\alpha$. Since
$\oR^*(\U)\geq\alpha$, there must be some witness collection $\W\mless \U$ of certified
Ramsey measures with $\oR^*(\W)=\beta$. By the inductive hypothesis applied to
$\M$, for all $W\in\W$, we have $N_W\models\oR^*(\kappa)\geq\beta$. Therefore,
since \(\W\mless\U\), there is for every
$A\subseteq\kappa$ in $N_U$ some \(A\)-good certified Ramsey measure $W\in N_U$
with $N_W\models\oR^*(\kappa)\geq\beta$ and so $N_U$, by collecting these together, has a
witness collection $\overline{\W}$ of certified Ramsey measures such that
$N_W\models\oR^*(\kappa)\geq\beta$ holds for all
$W\in\overline{\W}$. But then, by applying our inductive
hypothesis to $N_U$, we have that $\oR^*(\overline{\W})\geq\beta$ in $N_U$. This completes
the proof in one direction.

In the other direction, suppose that $\U$ is a witness collection of certified Ramsey
measures such that $N_U\models\oR^*(\kappa)\geq\alpha$ for all $U\in \U$. We must show that
for all $\beta<\alpha$, there is a witness collection $\W$ of certified Ramsey measures with
\(\W\mless\U\) and $\oR^*(\W)\geq\beta$. Fix $\beta<\alpha$. For each $U\in\U$, we can fix
some $\W_U$, which
$N_U$ thinks is a witness collection of certified Ramsey measures of rank at least $\beta$,
and let $\W=\Union_{U\in\U}\W_U$. Since all our measures are certified, $\W$ is a witness
collection of certified Ramsey measures and we have arranged that \(\W\mless\U\).
If $W\in \W$, then $W\in\W_U$ for some
$U\in \U$ and therefore, since $N_U\models\oR^*(\W_U)\geq\beta$, we get
$N_W\models\oR^*(\kappa)\geq\beta$ by applying the inductive
hypothesis to $N_U$. Thus, by the inductive hypothesis,
$\oR^*(\W)\geq\beta$, which establishes this direction.
\end{proof}

Note that the fact that the Ramsey measures we are working with are certified
only came into play in the second part of the proof. Essentially, being a Ramsey measure
is downward absolute. On the other hand, we would have run into trouble in the second
part if we had built the collection \(\W\) using only ordinary Ramsey measures, since
\(N_U\) and \(\M\) might disagree on whether a given filter is \(\omega_1\)-intersecting.
This observation will be important when we revisit this proof in
Theorem~\ref{th:rankMeasures}.

The desired result, which is the same lemma for witness collections of Ramsey
measures, will follow once we establish that $\oR^*(\kappa)=\oR(\kappa)$. First, we have to
review a few basic facts which we will use now and in later sections.

Observe that if $U$ is a Ramsey measure, then the intersection of any countably many sets in
$U$ has size $\kappa$ because if the intersection was bounded by $\alpha<\kappa$, we could
add $\kappa\setminus\alpha$ to the sets being intersected (recall that all tails of
\(\kappa\) are in $U$
by assumption) and violate the $\omega_1$-intersecting property. Another useful fact is that
for every ordinal $\kappa\leq \alpha<\kappa^+$, if $E$ is a well-ordering of $\kappa$ in
order-type $\alpha$, then there is a single function $g^E:\kappa\to\kappa$ such that
whenever $U$ is an $M$-ultrafilter with a well-founded ultrapower and $E\in M$, then
$g^E\in M$ and $[g^E]_U=\alpha$ in the ultrapower. We call such $g^E$ a
\emph{representing function} for $\alpha$ and can define it by simply letting $g^E(\xi)$
be the order-type of
$E\restrict\xi\times\xi$.

\begin{theorem}\label{th:rankMeasures}
Let \(\alpha\) be an ordinal and \(\M\) practical for \(\kappa\). Then the following
hold in \(\M\):
\begin{enumerate}
\item a witness collection \(\U\) of Ramsey measures on \(\kappa\) has \(\oR(\U)\geq\alpha\)
if and only if \(N_U\models\oR(\kappa)\geq\alpha\) for all \(U\in\U\);

\item if there is a witness collection \(\U\) of Ramsey measures on \(\kappa\) with
\(\oR(\U)\geq\alpha\), then there is also a witness collection \(\U^*\) of certified
Ramsey measures with \(\oR(\U^*)\geq\alpha\).
\end{enumerate}
\end{theorem}

\noindent Note that we used the simple M-rank in both cases in part (2) of the lemma
and not the M-rank restricted to witness collections of certified measures.

\begin{proof}
We will prove both parts of the statement simultaneously by induction on \(\alpha\).
For the base case \(\alpha=0\) part (1) is trivial and part (2) follows because, as we
already noted, if $\kappa$ is Ramsey, then for every $A\subseteq\kappa$, there is an $A$-good certified Ramsey measure. So suppose inductively that the statement holds for all
$0\leq\beta<\alpha$. Fix an \(\M\) practical for \(\kappa\) and work in \(\M\).

Let us first show that part (1) holds for \(\alpha\) by mirroring the proof of
Lemma~\ref{lem:rankCertifiedMeasures}.
The forward direction goes through exactly as in the proof of
Lemma~\ref{lem:rankCertifiedMeasures}, since, as we noted after that proof, the fact
that the Ramsey measures were certified played no part in this particular argument.
For the converse, suppose that \(\U\) is a witness collection of Ramsey measures
such that \(N_U\models\oR(\kappa)\geq\alpha\) for all \(U\in\U\). Fix a \(\beta<\alpha\).
For each \(U\in\U\) we can fix \(\W_U\) which \(N_U\) thinks is a collection of
Ramsey measures of rank at least \(\beta\) and let \(\W=\bigcup_{U\in\U}\W_U\). We would
like to say that \(\W\) is a witness collection of Ramsey measures, but this need not be
the case if we are working with arbitrary (noncertified) Ramsey measures. Instead, we
apply part (2) of the induction hypothesis to each \(N_U\) to replace each \(\W_U\)
with a witness collection \(\W_U^*\) of \emph{certified} Ramsey measures satisfying
\(\oR(\W_U^*)\geq\beta\) in \(N_U\). If we now let \(\W^*=\bigcup_{U\in\U}\W_U^*\),
this actually is a witness collection of Ramsey measures and, again, \(\W^*\mless\U\).
The rest of the argument
proceeds as before: if \(W\in\W^*\), then \(W\in \W_U^*\) for some \(U\in\U\) and
therefore, since \(N_U\models \oR(\W_U^*)\geq\beta\), we get
\(N_W\models\oR(\kappa)\geq\beta\) by applying part (1) of the induction hypothesis to
\(N_U\). Thus, by part (1) of the induction hypothesis again, \(\oR(\W^*)\geq\beta\).
Altogether, this shows that \(\oR(\U)\geq\alpha\).

Now we move on to show that part (2) holds for \(\alpha\).
Suppose that there is a witness collection $\U$ of Ramsey measures on $\kappa$ with
$\oR(\U)\geq\alpha$. We need to show that there is a witness collection $\U^*$ of certified
Ramsey measures on $\kappa$ with $\oR(\U^*)\geq\alpha$. By what we just argued it follows
that $N_U\models\oR(\kappa)\geq\alpha$ for all $U\in\U$. The next step is to replace each
$U$ with some $U^*$, where $U^*$ is certified and $N_{U^*}$ also satisfies that
$\oR(\kappa)\geq\alpha$. For this, we need to look more closely at how a good set $I$ of
indiscernibles for $L_\kappa[A]$ is constructed.

For every $A\subseteq\kappa$, there is an associated club $C_A$ in $\kappa$ and a regressive
function $f_A:[C_A]^{\lt\omega}\to\kappa$ such that any homogeneous set for $f_A$ is a good
set of indiscernibles for $L_\kappa[A]$. The club $C_A$ and function $f_A$ are defined
simply enough from $A$ that any transitive model of $\ZFC^-$ containing $A$ also contains
$C_A$ and $f_A$ (for details, see \cite{gitman:dissertation}, chapter 2). Given an $A$-good
Ramsey measure $U$, we will find  a homogeneous set $I$ of size $\kappa$ for $f_A$ by
showing that for each $n<\omega$, the restriction $f_n:[C_A]^n\to\kappa$ of $f_A$ has a
homogeneous set in $U$ and using the $\omega_1$-intersecting property of $U$. Since $U$ is
weakly amenable, we can define the finite product $M_U$-ultrafilters $U^n$ for $n<\omega$
(where $U^1=U$) and since all iterated ultrapowers of $U$ are well-founded and the
ultrapower by $U^n$ is isomorphic to the $n^{\text{th}}$-iterated ultrapower of $U$, it
follows that all ultrapower maps $j_{U^n}:M_U\to N_{U^n}$ are embeddings into transitive
models. Standard facts about product ultrafilters also tell us that a set
$B\subseteq\kappa^n$ is in $U^n$ if and only if
$\la \kappa,j_U(\kappa),j_{U^2}(\kappa),\ldots,j_{U^{n-1}}(\kappa)\ra\in j_{U^n}(B)$.
Now fix $n<\omega$ and consider $f_n$. The set $[C_A]^n$ is in $U^n$, since every club is in
$U$, and so we can let
$j_{U^n}(f_n)(\kappa,j_U(\kappa),j_{U^2}(\kappa),\ldots,j_{U^{n-1}}(\kappa))=\xi$,
where we must have $\xi<\kappa$ since $j_{U^n}(f_n)$ is regressive by elementarity. It
follows that the set
\begin{displaymath}
X_n=\{\la\xi_1,\xi_2,\ldots,\xi_n\ra\in [C]^n\mid f_n(\xi_1,\xi_2,\ldots,\xi_n)=\xi\}
\end{displaymath}
is in \(U\).
By properties of product ultrafilters, there is a set $\overline X_n\in U$ such that every
sequence $\la \xi_1,\xi_2,\ldots,\xi_n\ra\in [C]^n$ with $\xi_i\in \overline X_n$ is in
$X_n$. Clearly each $\overline X_n$ is homogeneous for $f_n$ and so we can intersect all the
$\overline X_n$ to obtain a homogeneous set $I$ of size $\kappa$ for $f_A$. Note that we can
further refine $I$ by adding some other sets in $U$ to the intersection.

Now fix some $A\subseteq\kappa$ and find an $\{A,E\}$-good Ramsey measure $U$ in $\U$,
where \(E\) is some well-order of \(\kappa\) of order-type $\alpha$, so that
we have the representing function \(g^E\) in \(M_U\) (if \(\alpha<\kappa\) we can use
a constant function instead of \(g^E\) and omit \(E\) from the following discussion).
Since $N_U\models\oR(\kappa)\geq\alpha$, the set
\begin{displaymath}
X=\{\xi<\kappa\mid \oR(\xi)\geq g^E(\xi)\}
\end{displaymath}
is in \(U\).
This is crucial to the ensuing construction. Let $A^*\subseteq\kappa$ code the triple
$\{A,E,V_\kappa\}$. Now we consider the regressive function
$f_{A^*}:[C_{A*}]^{\lt\omega}\to\kappa$ and construct a good set $I_{A^*}$ of indiscernibles
for $L_\kappa[A^*]$ by intersecting the sets $\overline X_n$, homogeneous for $f_n$,
together with $X$. This ensures that $I_{A^*}\subseteq X$. Using $I_{A*}$, we construct a
certified Ramsey measure $U^*$ with $A,E,V_\kappa\in M_{U^*}$, which is certified by
$I_{A^*}$. Note that $X$ is an element of $M_{U^*}$ because it is definable over $V_\kappa$
from $E$ and so it must be the case that $X\in U^*$. But since $E\in M_{U^*}$, it follows
that $[g^E]=\alpha$ in the ultrapower $N_{U^*}$ and so
$N_{U^*}\models\oR(\kappa)\geq\alpha$. Thus, we have succeeded in finding for every
$A\subseteq\kappa$, an $A$-good certified Ramsey measure $U^*$ such that
$N_{U^*}\models\oR(\kappa)\geq\alpha$. Let $\U^*$ be the witness collection consisting of
these $U^*$. By part (1) above, $\oR(U^*)\geq\alpha$, which completes the argument.
\end{proof}

The proofs of analogous results for witness collections of strong Ramsey or super Ramsey
measures are even easier, in that we do not even need to introduce certified measures.
A $\kappa$-model is always correct about a set being a strong
Ramsey measure and a $\kappa$-model that is elementary in $\her{\kappa}$ is always correct
about a set being a super Ramsey measure: if $M\prec\her{\kappa}$ is a $\kappa$-model and
$U$ is a weakly amenable $M$-ultrafilter, then $M$ is the $\her{\kappa}$ of the ultrapower
$N$ and therefore if $\overline M\prec\her{\kappa}^N$, then $\overline M\prec\her{\kappa}$.

A more direct approach to defining the rank of a Ramsey (or Ramsey-like) cardinal, without
introducing the order on the witness collections, would be as follows. Define that the
Ramsey rank of $\kappa$ is $0$ if $\kappa$ is not Ramsey, that the Ramsey rank of
$\kappa$ is $\geq 1$ if $\kappa$ is Ramsey, and now inductively that the Ramsey rank
of $\kappa$ is $\geq\alpha$ if for every $A\subseteq\kappa$ and $\beta<\alpha$, there is an
$A$-good Ramsey measure $U$ such that the Ramsey rank of $\kappa$ in $N_U$ is $\geq\beta$.
Finally, define that the rank of $\kappa$ is exactly $\alpha$ if it is $\geq\alpha$, but it
is not $\geq\alpha+1$. As a corollary of Theorem~\ref{th:rankMeasures}, we get that the
M-rank is precisely the Ramsey rank we just described.

\begin{corollary}\label{cor:approxMitchellRank}
If \(\alpha\) is an ordinal and \(\M\) is practical for \(\kappa\), then, in \(\M\),
we have $\oR(\kappa)\geq\alpha$ if and only if for every $A\subseteq\kappa$ and every
$\beta<\alpha$, there is an $A$-good Ramsey measure $W$ on $\kappa$ with
$N_W\models \oR(\kappa)\geq\beta$.
The same result holds for strongly Ramsey and super Ramsey measures.
\end{corollary}

\noindent Corollary~\ref{cor:approxMitchellRank} allows us to calculate $\oR(\kappa)$ inside
$\her{\kappa}$ and confirms the intuition that objects in $\her{\kappa}$ should suffice to
compute the rank of a Ramsey or Ramsey-like cardinal.
An important advantage of this alternative description of the \(M\)-rank is that it is
meaningful even in models of set theory which are not practical for \(\kappa\),
e.g.\ in \(\kappa\)-models where
\(\mathcal{P}(\kappa)\) does not exist. Such models might contain many Ramsey measures
on \(\kappa\) but cannot collect them into a witnessing collection. Consequently, the
M-rank, as originally defined, of \(\kappa\) in such a model would be 0, but computing it
in this alternative way might give nontrivial values.

Next, as promised, we show that there are always ``well-behaved" witness collections of
Ramsey (strongly Ramsey, super Ramsey) measures.

\begin{theorem}\label{th:wellBehavedCollections}
If \(\alpha\) is an ordinal and \(\M\) is practical for \(\kappa\), then, in \(\M\), whenever
$\oR(\kappa)>\alpha$, there is a witness collection $\U$ of Ramsey measures on $\kappa$
with $\oR(\U)=\alpha$ such that $N_U\models\oR(\kappa)=\alpha$ for all $U\in\U$. The same
result holds for strongly Ramsey and super Ramsey measures.
\end{theorem}

\begin{proof}
As usual, we prove the result for Ramsey measures. Fix an \(\M\) practical for $\kappa$ and
work in $\M$. Suppose that $\W$ is a witness collection of Ramsey measures on $\kappa$ with
$\oR(\W)=\alpha$. If for every $A\subseteq\kappa$, there is some $A$-good Ramsey measure $U$
such that $N_U\models\oR(\kappa)=\alpha$, then we can let $\U$ be the witness collection of
such Ramsey measures, one for every $A$, and by Theorem~\ref{th:rankMeasures}, we would have
$\oR(\U)=\alpha$. Thus, we can suppose towards a contradiction that there is some
$A\subseteq\kappa$ such that for every $A$-good Ramsey measure $U$, if
$N_U\models\oR(\kappa)\geq\alpha$, then $N_U\models\oR(\kappa)>\alpha$.
It follows from this assumption and Theorem~\ref{th:rankMeasures} that there is for
every \(B\subseteq\kappa\) a \(\{A,B\}\)-good Ramsey measure \(W\in\W\) with
\(N_W\models\oR(\kappa)>\alpha\). Let \(\W_0\) be a witness collection consisting
of one such Ramsey measure for every \(B\).
Thus, if
$W\in \W_0$, then $N_W$ has what it thinks is a witness collection of certified Ramsey
measures of rank greater than $\alpha$ by Theorem~\ref{th:rankMeasures} and so, in
particular, $N_W$ has, for any \(B\subseteq\kappa\) in \(N_W\), an $\{A,B\}$-good certified
measure $\overline W$ with the property that $N_{\overline W}\models\oR(\kappa)>\alpha$.
Let $\W_1$ be the witness
collection formed by putting together all such certified measures from all $N_W$ for
$W\in \W_0$. By construction $\W_1\mless \W_0$. But $\W_1$ has the same property as $\W_0$,
namely that for every $W\in \W_1$, $N_W\models\oR(\kappa)>\alpha$. Thus, we can repeat the
process to construct $W_2\mless W_1$ with the same property and in this way obtain a
descending infinite sequence in $\mless$, which is impossible.
\end{proof}

The theorem allows us to obtain the following sharpened version of
Corollary~\ref{cor:approxMitchellRank}.

\begin{corollary}\label{cor:MitchellRank}
If \(\alpha\) is an ordinal and $\M$ is practical for $\kappa$, then, in $\M$, we have
$\oR(\kappa)\geq\alpha$ if and only if every $A\subseteq\kappa$ and every $\beta<\alpha$, there
is an $A$-good Ramsey measure $W$ on $\kappa$ with $N_W\models \oR(\kappa)=\beta$. The same
result holds for strongly Ramsey and super Ramsey measures.
\end{corollary}

We end the discussion of basic properties of the M-order on witness collections by showing
that strongly Ramsey cardinals have the maximum Ramsey rank, super Ramsey cardinals have the
maximum strongly Ramsey rank and measurable cardinals have the maximum super Ramsey rank.

\begin{theorem}$\,$
\begin{enumerate}
\item If $\kappa$ is strongly Ramsey, then $\oR(\kappa)=\kappa^+$.
\item If $\kappa$ is super Ramsey, then $\oSR(\kappa)=\kappa^+$.
\item If $\kappa$ is measurable, then $\oSuR(\kappa)=\kappa^+$.
\end{enumerate}
\end{theorem}

\begin{proof}
Let us introduce an intermediate large cardinal property between Ramsey and super Ramsey
cardinals by removing the $\kappa$-model assumption from the definition of super Ramsey
cardinals. Call a cardinal \emph{weakly super Ramsey} if for every $A\subseteq\kappa$, there
is a weak $\kappa$-model $M\prec\her{\kappa}$ containing $A$ for which there is a weakly
amenable $\omega_1$-intersecting $M$-ultrafilter on $\kappa$. We will argue that a weakly
super Ramsey $\kappa$ must have maximum Ramsey rank. Suppose not, meaning that
$\oR(\kappa)=\alpha<\kappa^+$. By Corollary~\ref{cor:MitchellRank} there is in
$\her{\kappa}$, for every $\beta<\alpha$, an $A$-good Ramsey measure $W$ with
$N_W\models\oR(\kappa)=\beta$. Fix $A\subseteq\kappa$ and let $M\prec\her{\kappa}$ be a weak
$\kappa$-model containing $A$ and $\alpha$ for which there is a weakly amenable
$\omega_1$-intersecting $M$-ultrafilter on $\kappa$. Let $N$ be the ultrapower of $M$ by
$U$. If $\beta<\alpha$, then, by elementarity, $M$ satisfies that for every
$B\subseteq\kappa$, there is a Ramsey measure $W$ on $\kappa$ with
$N_W\models\oR(\kappa)=\beta$, and so $N$ must satisfy this as well. But then
Corollary~\ref{cor:MitchellRank} implies $N\models\oR(\kappa)\geq\alpha$. Thus, we have
shown that for every $A\subseteq\kappa$, there is an \(A\)-good Ramsey measure $U$ on
$\kappa$ with $N_U\models\oR(\kappa)=\alpha$, which means that $\oR(\kappa)>\alpha$,
contradicting our assumption.

Let $\kappa$ be strongly Ramsey. We will show that for every $\alpha<\kappa^+$, there is a
witness collection of Ramsey measures on $\kappa$ of rank $\alpha$. Fix $A\subseteq\kappa$
and let $U$ be an $\{A,\alpha\}$-good strong Ramsey measure on $\kappa$. We will now argue
that $\kappa$ is weakly super Ramsey in $N_U$. Fix $B\subseteq\kappa$ in $N_U$. Using the
construction from the proof of Lemma~\ref{lem:omegaspecial}, we obtain a sequence
$\la (M_n,W_n)\mid n<\omega\ra$ such that $B\in M=\Union_{n<\omega}M_n$ is elementary in
$M_U$ and $W=\Union_{n<\omega}W_n$ is a weakly amenable $M$-ultrafilter on $\kappa$. Since
$N_U$ is a $\kappa$-model both $M$ and $W$ are in $N_U$. Thus, we have verified that
$\kappa$ is weakly super Ramsey in $N_U$, and so it follows that
$N_U\models\oR(\kappa)=\kappa^+$, so in particular $N_U\models\oR(\kappa)\geq\alpha$. But
this means that for every $A\subseteq\kappa$, there is an $A$-good Ramsey measure $U$ with
$N_U\models\oR(\kappa)\geq\alpha$, from which it follows that there is a witness collection
of Ramsey measures on $\kappa$ of rank $\alpha$, as required.

To show that strong Ramsey cardinals have maximum strong Ramsey rank, we just mimic the
argument that weakly super Ramsey cardinals have maximum Ramsey rank. To show that
measurable cardinals have maximum super Ramsey rank, we use that measurable cardinals are
super Ramsey and repeat the same argument.
\end{proof}

\noindent Note that we didn't need that $\kappa$ is strongly Ramsey in the argument that
\(\oR(\kappa)=\kappa^+\), but merely that $\kappa$ is $\omega$-closed Ramsey
($M^\omega\subseteq M$), which gives a lower bound on the strength of having maximum Ramsey
rank. In fact, the proof shows that \(\omega\)-closed Ramsey cardinals are stationary limits
of Ramsey cardinals of maximal Ramsey rank.

\section{Extensions with cover and approximation properties cannot increase Ramsey or
Ramsey-like rank}

In \cite{hamkins:coverandapproximations}, Hamkins developed general techniques to show that
if $V\subseteq V'$ has the $\delta$-cover and $\delta$-approximation properties for some
regular cardinal $\delta$ of $V'$, then for most large cardinal properties $V'$ cannot have
new large cardinals of that type above $\delta$. The techniques cannot be applied directly
to Ramsey or Ramsey-like cardinals because, for the smaller large cardinals, they require
embeddings to exist for all transitive models of size $\kappa$ (as in the case of weakly
compact cardinals), and in particular for all $\kappa$-models, which we know is not the case
for Ramsey or Ramsey-like cardinals. Nevertheless, we will be able to adapt the machinery
used in the proofs of theorems in \cite{hamkins:coverandapproximations} to the situation of
our cardinals. We will show that if $V\subseteq V'$ has the $\delta$-cover and
$\delta$-approximation properties (for some regular $\delta$ of $V'$) and $\kappa>\delta$
has Ramsey (or Ramsey-like) rank $\alpha$ in $V'$, then it had at least rank $\alpha$ in
$V$. The significance of the result lies in applying it to forcing extensions to show that
no new Ramsey or Ramsey-like cardinals of any rank were created. Although it is easy to show
that Ramsey cardinals cannot be created by small forcing, it was not previously known
whether the result generalized to all extensions with cover and approximation properties. We
begin by recalling the definition of cover and approximation properties and their connection
to forcing extensions.

\begin{definition}[Hamkins \cite{hamkins:coverandapproximations}]
\label{def:coverandextprop}
Suppose $V\subseteq V'$ are transitive (set or class) models of (some fragment of) $\ZFC$
and $\delta$ is a cardinal in $V'$.
\begin{enumerate}
\item The pair $V\subseteq V'$ satisfies the $\delta$-\emph{cover property} if for every
$X\in V'$ with $X\subseteq V$ and $|X|^{V'}<\delta$, there is $Y\in V$ with $X\subseteq Y$
and $|Y|^V<\delta$.

\item The pair $V\subseteq V'$ satisfies the $\delta$-\emph{approximation property} if
whenever $X\in V'$ with $X\subseteq V$ and $X\cap x\in V$ for every $x$ of size less than
$\delta$ in $V$, then $X\in V$.
\end{enumerate}
\end{definition}

\noindent If $\p$ is a forcing notion of size at most $\delta$, then the pair
$V\subseteq V[G]$, where $G$ is $V$-generic for $\p$, has the $\delta^+$-cover and
$\delta^+$-approximation properties. We say that a poset $\p$ has a \emph{closure point} at
a cardinal $\delta$ if $\p$ factors as $\mathbb R*\dot \q$, where $\mathbb R$ is
nontrivial\footnote{Here, a poset is \emph{nontrivial} if it necessarily adds a new set.}
of size at most $\delta$ and
$\forces_\mathbb R \dot \q\text{ is strategically }\leq\delta\text{-closed}$. We then have:

\begin{theorem}[Hamkins]\label{th:gapcoverapprox}
If $\p$ is a forcing notion with a closure point at $\delta$, then the pair
$V\subseteq V[G]$ satisfies the $\delta^+$-cover and $\delta^+$-approximation properties for
any forcing extension $V[G]$ by $\p$.
\end{theorem}

\noindent Thus, we will be able to show that a large class of forcing notions, namely those
with a closure point less than or equal to the first inaccessible cardinal (or in fact much
higher), cannot create new Ramsey or Ramsey-like cardinals of any  rank.

\subsection{Strongly Ramsey and super Ramsey cardinals}
Most of the work in showing that strong Ramsey rank cannot increase in extensions with cover
and approximation properties goes into showing that strongly Ramsey cardinals cannot be
created in such extensions. Once again, to carry out the inductive arguments, we will need
the statements to be formulated in terms of practical models, with the hypothesis that
$\M$ and $\M'$ are practical for \(\kappa\) and that the pair $\M\subseteq \M'$ has the
$\delta$-cover and $\delta$-approximation properties and the same ordinals.

\begin{lemma}
\label{le:ApproxAndCoveringGoesDownInitialSegments}
Suppose the pair \(V\subseteq V'\) satisfies the \(\delta\)-cover and
\(\delta\)-approximation properties and has the same ordinals.
Let \(\lambda\geq\delta\) be a regular cardinal in \(V'\). Then the pair
\(H_\lambda^V\subseteq H_\lambda^{V'}\) also satisfies the \(\delta\)-cover and
\(\delta\)-approximation properties and has the same ordinals.
\end{lemma}

\begin{proof}
Suppose that $\alpha$ is an
ordinal in $H_\lambda^{V'}$, meaning that $\alpha<\lambda$. It follows that $\alpha$ has
size less than $\lambda$ in $V$ as well (otherwise, $\lambda$ would not be a cardinal in
$V'$), and hence $\alpha\in H_\lambda^V$, so \(H_\lambda^V\) and \(H_\lambda^{V'}\) have the
same ordinals. If $X\subseteq H_\lambda^V$ is a set of size less
than $\delta$ in $H_\lambda^{V'}$, then, by the \(\delta\)-covering property of
\(V\subseteq V'\) there is a $Y$ of size less than $\delta$ in $V$ with
$X\subseteq Y$. Let $\overline Y=Y\cap H_\lambda^V$. Then $\overline Y$ is a set of size
less than $\delta$ each of whose elements has transitive size less than $\lambda$, and so by
regularity of $\lambda$ we get $\overline Y\in H_\lambda^V$, showing that
\(H_\lambda^V\subseteq H_\lambda^{V'}\) satisfies the \(\delta\)-cover property.
Finally, suppose that
$X\subseteq H_\lambda$ is in $H_\lambda^{V'}$ and $x\cap X\in H_\lambda^V$ for all
$x\in H_\lambda^V$ of size less than \(\delta\) in \(V\). It follows from the
\(\delta\)-approximation property of \(V\subseteq V'\) that $X\in V$. Finally, $X$ has size
less than $\lambda$ in $V'$, but then it must have size less than $\lambda$ in $V$ as well,
showing that \(X\in H_\lambda^V\) and that \(H_\lambda^V\subseteq H_\lambda^{V'}\) satisfies
the \(\delta\)-approximation property.
\end{proof}

This lemma shows that we can deduce results about Ramsey and Ramsey-like ranks in
class-sized extensions with cover and approximation properties from corresponding results
about extensions of practical models, since we may always restrict to sufficiently large
\(H_\lambda\).

\begin{theorem}\label{th:noNewStrongOrSuperRamsey}
Suppose that $\M$ and $\M'$ are practical for \(\kappa\) and that
$\M\subseteq \M'$ has the $\delta$-cover and $\delta$-approximation properties for some
regular cardinal $\delta$ of $\M'$ and the same ordinals. If $\kappa>\delta$ is strongly
Ramsey in $\M'$, then $\kappa$ was already strongly Ramsey in $\M$. The same result holds
for super Ramsey cardinals.
\end{theorem}

\noindent The proof adapts techniques developed in \cite{hamkins:coverandapproximations} to
the embeddings characterizing strongly Ramsey cardinals. We will note in the course of the
argument where the constructions occurred in \cite{hamkins:coverandapproximations}.

\begin{proof}
First, suppose that $\kappa>\delta$ is strongly Ramsey in $\M'$. Fix $A\subseteq\kappa$ in
$\M$. We need to show that $\M$ has a $\kappa$-model $M$ containing $A$ and a weakly amenable
$M$-ultrafilter $U$ on $\kappa$. In $\M'$, let $W$ be an $\{A,V_\kappa^{\M}\}$-good strong
Ramsey measure and let $j:M_W\to N_W$ be the ultrapower by
$W$. Note that both $V_\kappa^{\M'}$ and $V_\kappa^{\M}$ are $\ZFC$-models since $\kappa$ is
inaccessible in $\M'$ and therefore also in $\M$.

\begin{claim}\label{cl:VkappaCoverApprox}
The pair $V_\kappa^{\M}\subseteq V_\kappa^{\M'}$ has the $\delta$-cover and $\delta$-approximation properties.
\end{claim}
\begin{proof}
This is just Lemma~\ref{le:ApproxAndCoveringGoesDownInitialSegments}.
%
\end{proof}

\noindent Thus, by elementarity, $N_W$ satisfies that the pair
$N=j(V_\kappa^{\M})\subseteq j(V_\kappa^{\M'})=N'$ has the $\delta$-cover and
$\delta$-approximation properties, and it is correct about this. Observe that
$N'=V_{j(\kappa)}^{N_W}$, and therefore is a $\kappa$-model in $\M'$ since $N_W$ is a
$\kappa$-model in $\M'$ (it is not difficult to see that the ultrapower of $\kappa$-model is
always a $\kappa$-model). The next several claims have the aim to conclude that $N$ and
$W\cap N$ are in $\M$.

\begin{claim}\label{cl:sameVkappa}
$V_\kappa^N=V_\kappa^{\M}$.
\end{claim}

\begin{proof}
If $X\in V_\kappa^{\M}$, then $X=j(X)\in j(V_\kappa^{\M})=N$. So $V_\kappa^{\M}\subseteq N$.
Conversely, if $X\in V_\kappa^N$, then note first that \(X\in M_W=\dom(j)\). This is because
\(X\in V_\kappa^{N_W}=V_\kappa^{M_W}\).
But then $j(X)=X\in N=j(V_\kappa^{\M})$, and so $X\in V_\kappa^{\M}$.
\end{proof}

\begin{claim}\label{cl:mutualCover}
If $X\subseteq \ORD^N$ is a set of size less than $\delta$ in $\M'$, then there is
$Y\in \M\cap N$ of size at most $\delta$ in $N'$ such that $X\subseteq Y$.
\end{claim}

\begin{proof}
This construction mimics Lemma~3.2 in \cite{hamkins:coverandapproximations}. Let $X_0=X$, and
observe that $X_0\in N'$ since $N'$ is a $\kappa$-model in $\M'$. So, by the $\delta$-cover
property of $N\subseteq N'$, there is $X_1\subseteq\ORD^N$ of size less than $\delta$ in $N$
such that $X_0\subseteq X_1$. Then, by the $\delta$-cover property of $\M\subseteq \M'$,
there is $X_2\subseteq\ORD^N$ of size less than $\delta$ in $\M$ such that $X_1\subseteq X_2$
(this uses that $\M$ and $\M'$ have the same ordinals). The set $X_2$ is in the
$\kappa$-model $N'$, and so, again, there is $X_3$ of size less than $\delta$ in $N$ such
that $X_2\subseteq X_3$. Continue bouncing between $N$ and $\M$ in this way. To get through 	
limit stages, observe that if $\gamma<\delta$ and $\la X_\xi\mid\xi<\gamma\ra$ is a sequence
of sets of size less than $\delta$ in $\M'$, then $X_\gamma=\Union_{\xi<\gamma}X_\gamma$ has
size less than $\delta$ in $\M'$ by the regularity of $\delta$. Thus, after $\delta$-many
steps, we end up with an increasing sequence $\la X_\xi\mid\xi<\delta\ra$ such that cofinally
many elements of it are in $N$ and cofinally many are in $\M$. Let
$Y=\Union_{\xi<\delta}X_\xi$. By closure, $Y\in N'$ and $Y$ has size at most $\delta$ there.
To see that \(Y\in N\) we use the \(\delta\)-approximation property of \(N\subseteq N'\).
Specifically, let \(y\in N\) have size less than \(\delta\). Then there is some
\(\xi<\delta\) such that \(Y\cap y=X_\xi\cap y\) and we may furthermore choose \(\xi\)
so that \(X_\xi\in N\). So then clearly \(Y\cap y\in N\) and we obtain \(Y\in N\)
by the \(\delta\)-approximation property. A similar argument, using the
\(\delta\)-approximation property of \(\M\subseteq\M'\), shows that also \(Y\in\M'\).
\end{proof}

\begin{claim}\label{cl:sameSmallSubsets}
$\M$ and $N$ have the same subsets of $\ORD^N$ of size less than $\delta$ in $\M'$.
\end{claim}

\begin{proof}
This argument mimics Lemma~3.3 in \cite{hamkins:coverandapproximations}. Suppose that
$X\subseteq \ORD^N$ has size less than $\delta$ in $\M'$. By Claim~\ref{cl:mutualCover},
there is a set $Y$ of size at most $\delta$ in $N'$ such that $X\subseteq Y$ and
$Y\in N\cap \M$. Let $Y=\{y_\alpha\mid\alpha<\gamma\}$ be the enumeration of $Y$ arising from
its order-type and note that $\gamma<{\delta^+}^{\M'}<\kappa$.
Since the order-type of $Y$ is absolute, the
enumeration is in both $N$ and $\M$. Let $\overline X=\{\alpha<\gamma\mid y_\alpha\in X\}$,
which is a subset of $\gamma$. Now observe that $X$ is in $N$ or $\M$ if and only if
$\overline X$ is there, and $\overline X$ is in $\M$ if and only if it is in $N$, by
Claim~\ref{cl:sameVkappa}, since it is a subset of $\gamma$.
\end{proof}

\begin{claim}\label{cl:expectedIntersection}
$\M\cap N'=N$ and $N\in \M$.
\end{claim}

\begin{proof}
This argument mimics Lemma~3.4 in \cite{hamkins:coverandapproximations}. First, we show
that $N\subseteq \M$. It suffices to verify that all sets of ordinals in $N$ are elements
of $\M$. Suppose that $X\subseteq \ORD^N$ is in $N$. Fix a set $x\subseteq\ORD^N$ of size
less than $\delta$ in $\M$. By Claim~\ref{cl:sameSmallSubsets} we have $x\in N$ and so
$X\cap x\in N$ as well. But $X\cap x$ is a set of ordinals in $N$ of size less than
$\delta$ in $\M'$, and therefore by Claim~\ref{cl:sameSmallSubsets} again, $X\cap x\in \M$.
So $X\in \M$ by the $\delta$-approximation property of $\M\subseteq \M'$.

Next, we verify that
$\M\cap N'\subseteq N$. Initially, we show that every set of ordinals in $\M\cap N'$ is in
$N$. So suppose that $X$ is a set of ordinals in $\M\cap N'$. Let $x$ be a set of ordinals of
size less than $\delta$ in $N$. Then $x\in \M$, by Claim~\ref{cl:sameSmallSubsets}, and so
$x\cap X$ in $\M$. By Claim~\ref{cl:sameSmallSubsets} again, $x\cap X\in N$. So, by the
$\delta$-approximation property of $N\subseteq N'$, we obtain $X\in N$.
Now suppose that $X$ is any set
in $\M\cap N'$. By $\in$-induction, suppose that every element of $X$ is in $N$. Since
$X\in N'$ and $N'$ is a model of $\ZFC$, there must be some ordinal $\beta$ in $N'$ such that
$X\in V_\beta^{N'}$ and thus $X\subseteq V_\beta^N=Y$. Enumerate
$Y=\{y_\alpha\mid \alpha<\gamma\}$ in $N$, and note that since $N\subseteq \M$, the
enumeration exists in $\M$ as well. Let $\overline Y=\{\alpha<\gamma\mid y_\alpha\in X\}$.
The set $\overline Y$ is in $\M$ and also in $N'$. So by what we already argued for sets of
ordinals, $\overline Y$ is in $N$, and hence so is $X$. This completes the argument that
$N=\M\cap N'$.

We will use the $\delta$-approximation property of $\M\subseteq \M'$ to argue that $N\in \M$.
Fix a set $x$ of size less than $\delta$ in $\M$. The intersection $N\cap x$ is in the 
$\kappa$-model $N'$, therefore there is some \(\beta\) such that \(N\cap x\in V_\beta^{N'}\).
It now follows, since obviously \(N\cap x\subseteq N\), that \(N\cap x\subseteq V_\beta^N\). Hence $N\cap x=V_\beta^N\cap x$ is in $\M$ since both $x$ and $V_\beta^N$ are there.
%
\end{proof}

Let $\overline U=N\cap W$.

\begin{claim}
$\overline U\in \M$.
\end{claim}

\begin{proof}
We will use the $\delta$-approximation property of $\M\subseteq \M'$. This construction
mimics Theorem~10 in \cite{hamkins:coverandapproximations}. Suppose that $x$ is a set of size
less than $\delta$ in $\M$. We can assume that $x\subseteq P(\kappa)^N$ and also that
whenever some $B\subseteq\kappa$ is in $x$, then so is the complement of $B$ in $\kappa$.
Since $W$ is an $M_W$-ultrafilter and $M_W$ is a $\kappa$-model in $\M'$, it follows that $W$
is $\kappa$-intersecting in $\M'$. So working in $\M'$, we consider the intersection of all
$B\in W\cap x$, which is non-empty, and hence must contain some element $\beta$. We will
argue that, for $B\in x$, we have $\beta\in B$ precisely when $B\in W$. By definition of
$\beta$, if $B\in W\cap x$, then $\beta\in B$. If $B\notin W$, then its complement $B^c$
is in $x\cap W$, and so $\beta\in B^c$, which means that $\beta\notin B$.
Thus, $\overline U\cap x=W\cap x=\{B\in x\mid \beta\in B\}$, which is clearly in $\M$.
\end{proof}

\begin{claim}
$N$ is closed under $\ltkappa$-sequences in $\M$, we have $A\in N$, and $\overline U$ is a
weakly amenable $\omega_1$-intersecting $N$-ultrafilter.
\end{claim}

\begin{proof}
Since $A\in N'=V_{j(\kappa)}^{N_W}$ and also $A\in \M$, we get \(A\in N\) by
Claim~\ref{cl:expectedIntersection}. If $\vec x=\la x_\xi\mid\xi<\gamma\ra$, for
some $\gamma<\kappa$, is a sequence of elements of $N$ in $\M$, then $\vec x$ is in both $\M$
and $N'$, and so $\vec x\in N$, again by Claim~\ref{cl:expectedIntersection}.

It is clear that $\overline U$ is an $\omega_1$-intersecting $N$-ultrafilter. It remains to
show that $\overline U$ is weakly amenable to $N$. Consider $X\cap \overline U$, where $X$ is
a set of size $\kappa$ in $N$. The set $X\cap \overline U$ is in $\M$ and also in $N'$ by the
weak amenability of $U$. Hence $X\in N$ by Claim~\ref{cl:expectedIntersection}.
\end{proof}

We return to the proof of the theorem. We now have, in \(\M\), a model \(N\), closed under
\(\ltkappa\)-sequences, with \(A\in N\) and a weakly amenable \(\omega_1\)-intersecting
\(N\)-ultrafilter. The only obstacle to seeing that \(\kappa\) is strongly Ramsey in
\(\M\) is that \(N\) might be too large in cardinality.
But this is easily fixed by building an elementary substructure $M$ of $\her{\kappa}^N$
containing $A$ in $\kappa$-many steps so that $M$ is a $\kappa$-model and
$U=\overline U\cap M$ is a weakly amenable $M$-ultrafilter.

To prove the same result for super Ramsey cardinals, we start with
$M_W\prec\her{\kappa}^{\M'}$. We will be done if we can argue that
$\her{\kappa}^N\prec \her{\kappa}^{\M}$. First, observe that we can use the
$\delta$-approximation property of $\M\subseteq \M'$ to define $\pset(\kappa)^{\M}$ in
$\her{\kappa}^{\M'}$, using the parameter $a=V_\kappa^{\M}$, as the collection of all
$X\subseteq\kappa$ such that for all $x\in a$, $x\cap X\in a$. Let's see that the same
formula defines $\pset(\kappa)^N$ in $\her{\kappa}^{N'}$. Recall that $a=V_\kappa^N\in N$ by
Claim~\ref{cl:VkappaCoverApprox}. If $X\subseteq\kappa$ is in $N$, then $x\cap X\in a$ for
every $x\in a$, and if $X\subseteq\kappa$ is in $N'$ and $x\cap X\in a$ for all $x\in a$, then $X\in N$ by the $\delta$-approximation property of $N\subseteq N'$.
Thus, using
the usual Mostowski coding, there is a formula $\psi(x,a)$ which defines $\her{\kappa}^{\M}$
in $\her{\kappa}^{\M'}$ and the same $\psi(x,a)$ also defines $\her{\kappa}^N$ in
$\her{\kappa}^{N'}$. Now suppose that
$\her{\kappa}^{\M}\models\exists x\,\varphi(x,b)$ for some $b\in \her{\kappa}^N$. So
$\her{\kappa}^{\M'}$ satisfies that $\exists x\,\varphi(x,b)$ holds in the collection defined
by $\psi(x,a)$. Since $\her{\kappa}^{N'}=M_W\prec\her{\kappa}^{\M'}$, it satisfies the same
statement, which gives that $\her{\kappa}^N\models\exists x\,\varphi(x,b)$.
\end{proof}

Next, we show that extensions with cover and approximation properties cannot increase the
rank of a strongly Ramsey or super Ramsey cardinal.

\begin{theorem}\label{th:rankDecreasesStrongRamsey}
Suppose that $\M$ and $\M'$ are practical for \(\kappa\) and that $\M\subseteq \M'$
has the $\delta$-cover and $\delta$-approximation properties for some regular cardinal
$\delta$ of $\M'$ and the same ordinals. If $\kappa>\delta$, then
$\oSR(\kappa)^{\M'}\leq \oSR(\kappa)^\M$. The same result holds for super Ramsey cardinals.
\end{theorem}

\begin{proof}
We will argue by induction on $\alpha$ that if $\M$ and $\M'$ are as in the
hypothesis and $\oSR(\kappa)=\alpha$ in $\M'$, then $\oSR(\kappa)\geq\alpha$ in $\M$. So
suppose inductively that the statement holds for all $\beta<\alpha$. Fix some pair
$\M$ and $\M'$ as in the hypothesis and suppose that $\oSR(\kappa)=\alpha$
in $\M'$. By Corollary~\ref{cor:approxMitchellRank} we have to show that
$\M$ has a strong Ramsey measure $U$ with $N_U\models\oSR(\kappa)\geq\beta$ for every
$\beta<\alpha$ and $A\subseteq\kappa$ in $\M$. So fix $\beta<\alpha$ and $A\subseteq\kappa$
in $\M$. In $\M'$ there is a strong Ramsey measure $W$ such that $M_W$ contains $A$ and
$V_\kappa^{\M}$ and $N_W\models\oSR(\kappa)=\beta$. In what follows we use the notation from
the proof of Theorem~\ref{th:noNewStrongOrSuperRamsey}. Construct $N$ and $N'$. Clearly
$N'\models\oSR(\kappa)=\beta$, since $N'=V_{j(\kappa)}^{N_W}$.
Since \(N\subseteq N'\) has the \(\delta\)-cover and \(\delta\)-approximation properties,
we may apply the induction hypothesis to this pair and conclude that
$N\models\oSR(\kappa)\geq\beta$. Corollary~\ref{cor:approxMitchellRank} now implies that
$\her{\kappa}^N$ satisfies that for every $B\subseteq\kappa$ and $\xi<\beta$ there is a
strong Ramsey measure $U_B$ with $N_{U_B}\models\oSR(\kappa)\geq\xi$ and so
$M\prec \her{\kappa}^N$ must satisfy this statement as well. It follows that the ultrapower
of $M$ by $U$ must, by weak amenability, satisfy the same statement, meaning that
$\oSR(\kappa)\geq\beta$ there, which is precisely what we set out to establish.
\end{proof}

\subsection{Ramsey cardinals}
The arguments presented in the previous section do not generalize directly to Ramsey
cardinals because we can no longer work with $\kappa$-models, whose properties were used
crucially in several places in the proof of Theorem~\ref{th:noNewStrongOrSuperRamsey} to pass
between $\M'$ and $N'$. Nevertheless, we can modify the proof to work for Ramsey cardinals
with the extra assumption that $\M'$ doesn't have new countable sequences of elements of
$\M$.

\begin{theorem}\label{th:noNewRamsey}
Suppose that $\M$ and $\M'$ are practical for \(\kappa\) and that
$\M\subseteq \M'$ has the $\delta$-cover and $\delta$-approximation properties for some
regular cardinal $\delta$ of $\M'$, they have the same ordinals and
$\M^\omega\subseteq \M$ in $\M'$.
If $\kappa>\delta$ is Ramsey in $\M'$, then $\kappa$ was already Ramsey in $\M$.
\end{theorem}

\begin{proof}
Fix some $A\subseteq\kappa$ in $\M$. In $\M'$, let $W$ be an $\{A,V_\kappa^{\M}\}$-good
Ramsey measure and let $j:M_W\to N_W$ be the ultrapower by $W$. We proceed at first as in the
proof of Theorem~\ref{th:noNewStrongOrSuperRamsey}. Let $N=j(V_\kappa^{\M})$ and
$N'=j(V_\kappa^{\M'})$. As before, the pair $N\subseteq N'$ has the $\delta$-cover and
$\delta$-approximation properties, but $N'$ is no longer a $\kappa$-model. Also, as before,
$V_\kappa^N=V_\kappa^{\M}$.

\begin{claim}\label{cl:sameSubsetsOfKappa}
$\pset^N(\kappa)=\M\cap \pset^{N'}(\kappa)$ and hence $\her{\kappa}^N\subseteq \M$.
\end{claim}

\begin{proof}
First, we show that $\pset^N(\kappa)\subseteq \M$ using the $\delta$-approximation property
of $\M\subseteq\M'$. So suppose that $B\subseteq \kappa$ is in $N$. Fix a set
$x\subseteq\kappa$ of size less than $\delta$ in $\M$, and note that
$x\in V_\kappa^{\M}\subseteq N$. Thus, $x\cap B\in V_\kappa^N\subseteq \M$ and the
\(\delta\)-approximation property gives \(B\in\M\).

Next, we show that subsets of $\kappa$ in the intersection of $\M$ and $N'$ must be in $N$
using the $\delta$-approximation property of $N\subseteq N'$. So suppose that
$B\subseteq\kappa$ and $B\in \M\cap N'$. Let $x\subseteq\kappa$ be a set of size less than
$\delta$ in $N$, meaning that $x\in V_\beta^{\M}$ for some $\beta<\kappa$. So $x\cap B$ is
also in $V_\beta^{\M}$, and hence is in $N$.

Finally, any \(X\in\her{\kappa}^N\) is coded by a subset of \(\kappa\) via Mostowski coding,
and this coding can be undone in \(\M\).
\end{proof}

We cannot prove that $N$ or even $\her{\kappa}^N$ is an element of $\M$. So instead we will
find a weak $\kappa$-model $M\prec \her{\kappa}^N$ for which $U=M\cap W$ is a weakly amenable
$M$-ultrafilter so that both $M$ and $U$ are in $\M$.

First, we argue that $\overline W=W\cap N$ is weakly amenable to $N$. Let $S$ be a subset of
$\pset(\kappa)$ of size $\kappa$ in $N$. By weak amenability $\overline S=S\cap W$ is in
$N'$. Now we will use the $\delta$-approximation property of $N\subseteq N'$ to get
$\overline S$ into $N$. Let $x\subseteq S$ be a set of size less than $\delta$ in $N$. We can
assume that whenever $B\subseteq\kappa$ is in $x$, then so is the complement of $B$ in
$\kappa$. Since $W$ is an $N_W$-ultrafilter and $\overline S\cap x$ is in $N_W$, it follows
that there is some $\beta$ that is an element of every $B\in \overline S\cap x$. Thus, the
sets in $\overline S\cap x$ are precisely the sets in $x$ having $\beta$ as an element, and
so $\overline S\cap x$ is in $N$.

Now we build $M\prec\her{\kappa}^N$, working in $\M'$, as in the proof of
Lemma~\ref{lem:omegaspecial} from the sequence $\la (M_n,U_n)\mid n<\omega\ra$, so that
$M=\Union_{n<\omega}M_n$ and $U=\Union_{n<\omega} U_n\subseteq W$ is a weakly amenable
$\omega_1$-intersecting $M$-ultrafilter. Since each $M_n$ and $U_n$ are in
$\her{\kappa}^N\subseteq \M$, it follows by our closure assumption that $M$ and $U$ are in
$\M$.
\end{proof}

\begin{theorem}\label{th:rankDecreasesRamsey}
Suppose that $\M$ and $\M'$ are practical for \(\kappa\) and that $\M\subseteq \M'$
has the $\delta$-cover and $\delta$-approximation properties for some regular cardinal
$\delta$ of $\M'$, they have the same ordinals and $\M^\omega\subseteq \M$ in \(\M'\). If
$\kappa>\delta$, then $\oR(\kappa)^{\M'}\leq \oR(\kappa)^\M$.
\end{theorem}

\noindent The proof is identical to Theorem~\ref{th:rankDecreasesStrongRamsey}.

\begin{question}
Can we remove the assumption that $\M^\omega\subseteq \M$ in $\M'$ from
Theorem~\ref{th:rankDecreasesRamsey}?
\end{question}

\subsection{$\alpha$-iterable cardinals}
For completeness, we will argue here that extensions with cover and approximation properties
cannot create new $\alpha$-iterable cardinals provided that the extension has no new
countable sequences from the old model.

\begin{theorem}\label{th:noNewIterable}
Suppose that \(\M\) and \(\M'\) are practical for \(\kappa\) and that
\(\M\subseteq\M'\) has the \(\delta\)-cover and \(\delta\)-approximation properties
for some regular cardinal \(\delta\) of \(\M'\), they have the same ordinals and
\(\M^\omega\subseteq\M\) in \(\M'\). If \(\kappa>\delta\) is \(\alpha\)-iterable in
\(\M'\), then it was already \(\alpha\)-iterable in \(\M\).
\end{theorem}

\begin{proof}
Fix $A\subseteq\kappa$ in $V$. Working in $V'$, find an
\(\{A,V_\kappa^{\M},V_\kappa^{\M'}\}\)-good weak $\kappa$-model $M_W$ for which there is an
$\alpha$-iterable $M_W$-ultrafilter $W$. We follow the proof of Theorem~\ref{th:noNewRamsey}
exactly by considering the pair $N\subseteq N'$. We can show that
\(\mathcal{P}^N(\kappa)=\M\cap\mathcal{P}^{N'}(\kappa)\)
and that \(\overline{W}=W\cap N\) is weakly amenable to \(N\) exactly as there.
Working in \(\M'\) we then build the sequence $\la(M_n, U_n)\mid n<\omega\ra$ of elementary
submodels of \(\her{\kappa}^N\) and filters, and this sequence must be in $V$ by our
assumption. Thus, $M=\Union_{n<\omega}M_n$ and $U=\Union_{n<\omega}U_n$ are in $V$.
Finally, Lemma~3.8 of \cite{gitman:welch} implies that $U$ is $\alpha$-iterable.
\end{proof}

\begin{question}
Can we remove the assumption that $\M^\omega\subseteq \M$ in $\M'$ from
Theorem~\ref{th:noNewIterable}?
\end{question}

\section{Killing the M-rank softly}
We can use forcing to softly kill the rank of a Ramsey or Ramsey-like cardinal, meaning that,
if $\kappa$ has rank $\alpha$ and $\beta<\alpha$, then there is a cofinality preserving
forcing extension in which $\kappa$ has rank $\beta$. Let's consider the case of Ramsey
cardinals. We will obtain the desired forcing extension by carefully adding a club through
ordinals $\delta$ with $\oR(\delta)< g^E(\delta)$ (where $g^E$ is a representing function for
$\alpha$), while preserving $\oR(\kappa)\geq\beta$. The result will follow because no weak
$\kappa$-model containing such a club can have its ultrapower satisfy $\oR(\kappa)=\alpha$.

Recall that, if \(U\) is a Ramsey or Ramsey-like measure on \(M_U\) with the ultrapower
map \(j\colon M_U\to N_U\), then \(M_U^*=V_{j(\kappa)}^{N_U}\).

\begin{lemma}\label{le:orderOmegaSpecial}
Suppose $\oR(\kappa)=\alpha>0$. Then for every $A\subseteq\kappa$ and $\beta<\alpha$, there
is an $A$-good Ramsey measure $U$ such that $M_U$ is $\omega$-special and
$N_U\models\oR(\kappa)=\beta$, and hence $M_U^*$ is $\omega$-special and
$M_U^*\models\oR(\kappa)=\beta$.
\end{lemma}

\begin{proof}
Fix $A\subseteq\kappa$ and let $\overline U$ be any $A$-good Ramsey measure with
$N_{\overline U}\models\oR(\kappa)=\beta$. Following the proof of
Lemma~\ref{lem:omegaspecial}, we construct, in $\omega$-many steps, an $A$-good Ramsey
measure $U$ such that $M_U\prec M_{\overline U}$ is $\omega$-special and $\beta\in M_U$. By
Corollary~\ref{cor:approxMitchellRank} $M_{\overline U}$ has a $B$-good Ramsey measure $W$
with $N_W\models\oR(\kappa)\geq\xi$ for every $\xi<\beta$ and $B\subseteq \kappa$ in
$M_{\overline U}$ and \(\beta\) is the largest ordinal for which this is true. Thus, by
elementarity, $M_U$ has a $B$-good Ramsey measure $W$ with $N_W\models\oR(\kappa)\geq\xi$ for
every $\xi<\beta$ and $B\subseteq \kappa$ in $M_U$ and \(\beta\) is still the largest ordinal
for which this is true. It follows that $N_U\models\oR(\kappa)=\beta$.

Recall from Section~\ref{sec:prelim} that whenever $M_U$ is $\omega$-special, then so is
$M_U^*$. Since $M_U^*=V_{j(\kappa)}^{N_U}$, where $j$ is the ultrapower map by $U$, it
satisfies $\oR(\kappa)=\beta$.
\end{proof}

\begin{theorem}
If $\oR(\kappa)=\alpha$ and $\beta<\alpha$ is any ordinal, then there is a cofinality
preserving forcing extension in which $\oR(\kappa)=\beta$. The same result holds for strongly
Ramsey and super Ramsey cardinals.
\end{theorem}

\begin{proof}
Suppose $\oR(\kappa)=\alpha$ and fix $\beta<\alpha$.
Since $\beta<\kappa^+$, we can fix some well-ordering $E$ of $\kappa$ in
order-type $\beta$ and let $g^E:\kappa\to \kappa$ be a representing function for
$\beta$ (see the discussion preceding Theorem~\ref{th:rankMeasures}; if \(\beta<\kappa\) we
can let \(\beta\) be represented by a constant function and omit \(E\) and \(g^E\) from the
following argument).

Let $\p_\kappa$ be the $\kappa$-length Easton support iteration, forcing at each inaccessible
$\gamma$ with $\q_\gamma$ to shoot a club, by closed initial segments, through the set of
cardinals $\delta<\gamma$ with $\oR(\delta)<g^E(\delta)$, and using trivial forcing
everywhere else. It is easy to see that each $\q_\gamma$ is $\lt\gamma$-strategically closed.
Fixing $\beta<\gamma$, the strategy to ensure that the union of a $\beta$-sequence of
conditions in $\q_\gamma$ with the supremum added on is itself a condition in $\q_\gamma$ is
to make sure that the supremum gets above $\beta$. This ensures that the supremum is not
inaccessible and so trivially has the property $\oR(\delta)<g^E(\delta)$. The forcing we
shall use to achieve our goal is $\p=\p_\kappa*\dot\q_\kappa$. This poset preserves all
cardinals and cofinalities, since each $\dot\q_\gamma$ is $\lt\gamma$-strategically closed in
$V^{\p_\gamma}$. Let $G*g\subseteq \p$ be $V$-generic.

The iteration $\p_\kappa$ has size $\kappa$ and the $\kappa$-cc
(cf.~\cite{cummings:handbook}) and elements of $\dot\q_\kappa$ are names for bounded subsets
of $\kappa$. Since each such name can be associated with a bounded subset of $\kappa$ by a
nice-name argument, we can assume that $\p\subseteq V_\kappa$. This means in particular that
every $A\subseteq\kappa$ in $V[G][g]$ has a $\p$-name $\dot A$ in $\her{\kappa}$ and so $A$
is an element of every model $M[G][g]$ where $M$ is a weak $\kappa$-model in $V$ containing
$\p$ and $\dot A$. The following claim will show that the rank of $\kappa$ in $V[G][g]$ is
still at least $\beta$.

\begin{claim}
If $\M\models\ZFC$ is a weak $\kappa$-model such that $V_\kappa,\beta,E\in \M$ and $\oR(\kappa)^{\M}<\beta$, then $\oR(\kappa)^{\M[G][g]}\geq\oR(\kappa)^{\M}$.
\end{claim}

\begin{proof}
Note that $\p\in \M$ since it is definable from $V_\kappa$ and $E$. We shall argue by
induction on $\xi<\beta$ that if $\M$ is as in the hypothesis and $\oR(\kappa)^{\M}=\xi$,
then $\oR(\kappa)^{{\M}[G][g]}\geq \xi$. So suppose inductively that the statement holds for
all $\eta<\xi$. Fix some $\M$ as in the hypothesis and suppose that $\oR(\kappa)^{\M}=\xi$.
For $A\subseteq\kappa$ in $\M[G][g]$ and $\eta<\xi$, we need to produce an $A$-good Ramsey
measure $W$ such that $N_W\models\oR(\kappa)\geq\eta$.

Let $A\in \M[G][g]$ be a subset of $\kappa$ and choose a $\p$-name $\dot{A}\in \M$ for it.
Fix $\eta<\xi$. We work in $\M$. By Lemma~\ref{le:orderOmegaSpecial} we can find an
\(\{\dot{A},V_\kappa,\beta,E\}\)-good Ramsey measure $U$ such that $M_U$ is $\omega$-special
and $N_U\models\oR(\kappa)=\eta$. Let $M=M_U^*$ and $h:M\to N$ be the ultrapower by $U$. Note
that $M\models\oR(\kappa)=\eta$, and so the inductive assumption applied to $M$ gives that
$M[G][g]\models\oR(\kappa)\geq\eta$. We shall lift $h$ to $M[G][g]$ in $\M[G][g]$ and argue
that the $M[G][g]$-ultrafilter by which the lift is the ultrapower is the desired $W$.

First, we lift $h$ to $M[G]$. To do this we need to find an $N$-generic filter for
$h(\p_\kappa)\cong\p_\kappa*\dot \q_\kappa*\dot\p_\tail$ containing $h\image G=G$. We will
use the filter $G*g$ for the $\p_\kappa*\dot\q_\kappa$ part of $h(\p_\kappa)$. Note that
$\p_\tail=(\dot\p_\tail)_{G*g}$ is $\lesseq\kappa$-strategically closed and hence
$\lesseq\kappa$-distributive in $N[G][g]$. Thus, by Lemma~\ref{le:diag2}, $\M[G][g]$ has an
$N[G][g]$-generic for $G_\tail$ for $\p_\tail$, and so we can lift $h$ to
$h:M[G]\to N[h(G)]$, where $h(G)=G*g*G_\tail$.

Next, we lift $h$ fully to $M[G][g]$ by finding an $N[h(G)]$-generic filter for
$h(\q_\kappa)$, where $\q_\kappa=(\dot\q_\kappa)_G$, containing $h\image g$. Let $C=\Union g$
and $\overline C=C\union \{\kappa\}$, which is in $N[h(G)]$ by our choice of $h(G)$. Note
that $\overline C$ is a closed bounded subset of $h(\kappa)$. If we can show that
$\overline C$ is a condition in $h(\q_\kappa)$, then we can use it as a master condition for
the lift and use Lemma~\ref{le:diag2} to find an $N[h(G)]$-generic filter $g^*$ for
$h(\q_\kappa)$ containing $\overline C$. The only reason it might not be the case that
$\overline C$ is an element of $\q_\kappa$ is that $\oR(\delta)^{N[h(G)]}\geq g^E(\delta)$
for some \(\delta\in C\) or $\oR(\kappa)^{N[h(G)]}\geq h(g^E)(\kappa)=\beta$.

The first option cannot occur, since otherwise we would get
$\oR(\delta)^{M[G]}\geq g^E(\delta)$ by elementarity. This contradicts the construction of
\(C\) which is a club of ordinals \(\delta\) satisfying \(\oR(\delta)^{M[G]}<g^E(\delta)\).

To see that the second option above also cannot occur, observe that the forcing
$h(\p_\kappa)$ has a closure point at the first inaccessible cardinal $\delta_0$: the first
non-trivial forcing happens at stage $\delta_0$ and has size $\delta_0$ and each
$\dot\q_\delta$ for $\delta>\delta_0$ is $\lesseq\delta_0$-strategically closed in
$V^{\p_\delta}$, from which it will follow that the remainder of the iteration is
$\lesseq\delta_0$-strategically closed. By Theorem~\ref{th:gapcoverapprox} the pair
$N\subseteq N[h(G)]$ has the $\delta_0^+$-cover and $\delta_0^+$-approximation properties.
Also $h(\p_\kappa)$ is clearly countably closed, meaning that $N^\omega\subseteq N$ in
$N[h(G)]$. Following our assumptions, $N\models\oR(\kappa)=\eta$ and hence
Theorem~\ref{th:rankDecreasesRamsey} yields that $N[h(G)]\models\oR(\kappa)\leq \eta<\beta$. This
completes the argument that \(\overline C\) is a condition in \(h(\q_\kappa)\), allowing us
to lift $h$ to $h:M[G][g]\to N[h(G)][g^*]$, where $g^*$ is obtained using
Lemma~\ref{le:diag2} below the master condition $\overline C$.

Next, we argue that the lift $h$ is the ultrapower by a weakly amenable
$\omega_1$-intersecting $M[G][g]$-ultrafilter $W$ in $\mathcal M[G][g]$. By
Lemma~\ref{le:ramseycountablyclosed}, since $\p$ is countably closed, $W$ is
$\omega_1$-intersecting. To conclude that $W$ is weakly amenable, we verify that $M[G][g]$
and $N[h(G)][g^*]$ have the same subsets of $\kappa$. Suppose $B$ is a subset of $\kappa$ in
$N[h(G)][g^*]$. Since $\p_\tail*h(\dot{\q}_\kappa)$ is $\leq\kappa$-distributive,
$B\in N[G][g]$ and so $B$ has a $\p$-name $\dot B\in N$, which we can take to be an element
of $\her{\kappa}^N=M$. So finally, $B\in M[G][g]$.

Recall that $M[G][g]\models\oR(\kappa)\geq\eta$, from which it follows that
$N_W\models\oR(\kappa)\geq\eta$ as well. This finishes the inductive argument and allows us
to conclude that $\M[G][g]\models\oR(\kappa)\geq \xi$.
\end{proof}

To see that $\oR(\kappa)^{V[G][g]}\leq\beta$, recall that
$C=\bigcup g$ is a club in $\kappa$, consisting of cardinals $\delta$
with $\oR(\delta)^{V[G]}<g^E(\delta)$. But since \(\q_\kappa\) is \(\ltkappa\)-distributive,
it also follows that \(\oR(\delta)^{V[G][g]}<g^E(\delta)\) for all \(\delta\in C\).
This means that in $V[G][g]$ there cannot be a
Ramsey measure $U$ with $C,E\in M_U$ and $N_U\models \oR(\kappa)\geq j(g^E)(\kappa)=\beta$.

Exactly the same argument would work to get the result for strongly and super Ramsey
cardinals, except that we would rely on Theorem~\ref{th:rankDecreasesStrongRamsey}.
\end{proof}

\bibliography{web}
\bibliographystyle{alpha}
\end{document}